\documentclass[11pt]{amsart}

\usepackage{epsf}
\usepackage{geometry}
\usepackage{listings} 
\usepackage{algorithmic,algorithm}
\usepackage{multirow}
\usepackage{enumerate}
\usepackage{booktabs}
\usepackage {ulem}
\usepackage {cancel}

\usepackage{accents}
\usepackage[pdftex]{graphicx}
\usepackage{amscd}
\usepackage[pdftex]{color} % black,white,red,green,blue,cyan,magen ta,yellow
\usepackage[pdftex,colorlinks]{hyperref}
\usepackage{lscape}
\usepackage{indentfirst}
\usepackage{latexsym}
\usepackage{amsmath, amsfonts, amssymb,mathrsfs}
\usepackage{verbatim}
\usepackage{bm}

\usepackage{amsmath}
\usepackage{todonotes} 
\usepackage[extra]{tipa}

\newcommand{\vertiii}[1]{{\left\vert\kern-0.25ex\left\vert\kern-0.25ex\left\vert #1 
    \right\vert\kern-0.25ex\right\vert\kern-0.25ex\right\vert}}

\hypersetup{
    bookmarks=true,         % show bookmarks bar?
    unicode=false,          % non-Latin characters in AcrobatâÃ¯s bookmarks
    pdftoolbar=true,        % show AcrobatâÃ¯s toolbar?
    pdfmenubar=true,        % show AcrobatâÃ¯s menu?
    pdffitwindow=true,      % page fit to window when opened
    pdftitle={},    % title
    pdfauthor={Johannes Kraus},     % author
    pdfsubject={},   % subject of the document
    pdfnewwindow=true,      % links in new window
    pdfkeywords={}, % list of keywords
    colorlinks=true,       % false: boxed links; true: colored links
    linkcolor=red,          % color of internal links
    citecolor=blue,        % color of links to bibliography
    filecolor=magenta,      % color of file links
    urlcolor=cyan           % color of external links
}

\everymath{\displaystyle}
\newtheorem{theorem}{Theorem}

\theoremstyle{remark}
\newtheorem{remark}{Remark}

\theoremstyle{definition}

\def\sig{{\bm \sigma}}
\def\btau{{\bm \tau}}

\def\bu{{\boldsymbol u}}
\def\bv{{\boldsymbol v}}
\def\bw{{\boldsymbol w}}
\def\bff{{\boldsymbol f}}

\def\tbx{{\bx}}
\def\tsig{\sig}
\def\tbu{{\bu}}
\def\tbw{{\bw}}
\def\tp{{p}}

\renewcommand{\div}{\operatorname{div}}

\renewcommand{\div}{\operatorname{div}}
\def\eps{\boldsymbol \epsilon}
\def\sig{\boldsymbol \sigma}

\def\P{\mathcal P}

\def\bf{\boldsymbol f}

\def\bu{\boldsymbol u}
\def\bv{\boldsymbol v}

\def\bw{\boldsymbol w}

\def\bx{\boldsymbol x}
\def\by{\boldsymbol y}
\def\bz{\boldsymbol z}

\def\bn{\boldsymbol n}

\def\divv{\text{div}}

\begin{document}

%\title[Mixed-mixed conservative finite element methods for Biot's model]
%{A family of mixed-mixed strongly conservative finite element methods for Biot's model of consolidation}
\title[Mixed-mixed strongly symmetric and mass-conservative FEM for Biot's model]
{A new family of mixed-mixed strongly symmetric and mass-conservative finite element methods for Biot's model}
% in four-field formulation}

\author{Qingguo Hong}
\address{Email*: qingguohong@mst.edu\\
Department of Mathematics and Statistics, Missouri University of Science and Technology, Rolla, MO 65409, USA
}
\author{Johannes Kraus}
\address{Email: johannes.kraus@uni-due.de\\
Faculty of Mathematics, University of Duisburg-Essen, Thea-Leymann-Straße 9, Essen 45127, Germany
}
\author{Maria Lymbery}
\address{Email: maria.lymbery@uni-due.de\\
Faculty of Mathematics, University of Duisburg-Essen, Thea-Leymann-Straße 9, Essen 45127, Germany
}
%\date{}


\maketitle

\begin{abstract}
We consider a four-field formulation of Biot's quasi-static model of consolidation with the (effective) elastic stress, the solid displacement,
the fluid flux and the fluid pressure as unknown physical quantities of interest. The weak form of this system composed of
the momentum and mass balance equations complemented by the stress-strain relation and Darcy's law yields a two-fold perturbed saddle-point problem. 
Its well-posedness is proven for a natural choice of inifinite-dimensional Hilbert spaces, where the displacement-flux pair is sought in a novel Hilbert space.
Next, based on recent works on $H(\rm div \, \rm div)$-conforming finite elements, we propose a family of
%$H({\rm div} \, {\rm div}) \cap H({\rm div})$-conforming
mixed-mixed finite element methods that preserve the angular momentum as well as the fluid mass balance pointwise, i.e., in a strong sense.
We prove the well-posedness of the arising discrete problem and establish optimal a priori error estimates for the related family of conservative mixed-mixed finite element methods.

\end{abstract}

\section{Introduction}\label{sec:introduction} 

The basic theory of poroelasticity describes the interaction between fluid flow and solid deformation within a linear porous
medium and was introduced by M.~A. Biot first for quasi-static phenomena~\cite{Biot1941general} in 1941 and later extended
to dynamics~\cite{Biot1956waves1,Biot1956waves2,Biot1962acoust,Biot1962Mech},  the latter works taking into account
wave propagation, energy dissipation, and other relaxation effects.
  
Traditionally, Biot's theory has been used in geomechanics, e.g., for monitoring and modeling land subsidence due to hydrocarbon
recovery or in reservoir engineering, see, e.g.,~\cite{Barenblatt1960basic,Zenisek1984finite}.
The important observation that biological tissues can be interpreted as porous, permeable and deformable media infiltrated
by fluids, such as blood and interstitial fluid, cf.~\cite{Bothl2022iterative,Pitre2017},
ultimately led to the successful  employment of the classical Biot's models and their generalizations in the studies of the biomechanical
behavior of human organs like the brain, 
\cite{Guo_etal2018subject-specific,
%Vardakis2019fluid,
KrLeLyOsSc2022}, 
the liver,~\cite{Aichele2021fluids}, the kidney,~\cite{Ong2009Deformation}, or the lung,~\cite{Dai2014comparison}.
      
Some works on the well-posedness of the mathematical models proposed by Biot and their stable discretization have been addressed
%in~\cite{Carlson1973Lin,Dafermos1968exist,Showalter2000Diff,Zienkiewicz1982Basic,Zienkiewicz1984dynamic,Zenisek1984existence,Mielke2013Homo}, 
%see also~\cite{Cheng2016Poro,Seifert2022Evol} for a more comprehensive overview.
in~\cite{Carlson1973Lin,Showalter2000Diff,Zienkiewicz1982Basic,Zenisek1984existence}, for a comprehensive overview
see~\cite{Cheng2016Poro}. Parameter robust well-posedness for various formulations of Biot's model can also be found in \cite{hong2023new} under a new framework.

Especially in applications targeting the modelling of biological tissue one faces physical parameter regimes in which
numerical approximation schemes potentially suffer from volumetric locking and/or spurious, nonphysical pressure oscillations. While the former has the same cause as in linear elasticity and can be cured by the use of (Stokes) stable discretizations,
%see e.g.
cf.~\cite{Babuska1992locking}, the latter may be understood as a failure of compatibility between the finite
element (FE) spaces and is related to violating the principle of mass conservation~\cite{Haga2012causes}.
Usually, both of these conditions can be observed when the displacement of the porous skeleton is driven
to a divergence-free state, the permeability of the porous solid is low and the so-called “constrained specific
storage constant” is close to zero~\cite{Phillips2009overcoming}, properties typically encountered in soft
tissue~\cite{Vardakis2016investigating,lee2018mixed}. Exactly divergence-free hybrid discontinuous Galerkin
(DG) methods are known for their favourable properties in computing approximate solutions of unsteady
incompressible flow problems, see~\cite{Lehrenfeld2016high}.

In~\cite{HongKraus2018parameter,KanschatRiviere2018finite}, a family of strongly mass-conserving
discretizations of Biot's model has been suggested, exploiting $H({\rm div})$-conforming DG approximations
of the displacement field, resulting in uniformly well-posed discrete models and parameter-robust near-best
approximations in parameter-dependent norms.  Generalizations of the strongly mass-conserving
discretizations of Biot's model to several fluid networks can be found in 
\cite{Hong2019conservativeMPET,Hong2020ParameterUzawa,HongEtAl2020parameter}.

Conservative discretizations in general aim at avoiding a violation of physical laws or principles as the
conservation of mass and the conservation of angular momentum in case of the quasi-static Biot model.
Symmetric stress approximations are desirable with regard to preserving the angular momentum. A mixed
formulation of both the mechanics and flow sub problems provides a suitable starting point to construct 
strongly conservative mixed-mixed finite element methods for the four-field model with the stress, the
displacement, the fluid flux, and the fluid pressure as physical quantities of interest. Such a model and
its conservative discretization are subject of the present work.
    
Existing works that investigate similar approaches to the numerical solution of the quasi-static Biot model
include the four-field formulation proposed in~\cite{Yi2014convergence}, which combines a mixed formulation
of the mechanics subproblem based on the Hellinger-Reissner principle with a standard mixed formulation of
the flow subproblem. The a priori error estimates in~\cite{Yi2014convergence} have been proved for a mixed
finite element method employing the Arnold-Winther and the Raviart-Thomas spaces for the discretization of
these two coupled subproblems, respectively.  
This model has further been investigated in~\cite{Lee2016robust} presenting error estimates, robust with respect
to the Lam\'{e} parameters and the limiting case of vanishing constrained storage coefficient, in $L^\infty$ norm
in time and $L^2$ norm in space.
With the skew-symmetric part of the gradient of the displacement as an additional unknown, which acts as
Lagrange multiplier to enforce the symmetry of the stress tensor, the method studied herein also works with
weakly symmetric stress approximations, thus offering certain advantages regarding computational cost and
ease of implementation, not least with regard to hybridization, see,
e.g.,~\cite{arnold1985mixed,CockburnGopalakrishnan2004characterization}.
Other three-, four -, and five-field models additionally involve a total pressure. Their discretization, a priori error
analysis, and robust preconditioning techniques have been addressed, e.g.,
in~\cite{oyarzua2016locking,Lee2017parameter,
%Kumar2020conservative,
Lee2023analysis}.

Finite element methods featuring symmetric stress approximations have been explored in the context of discretizing the equations of linear elasticity  \cite{arnold2002mixed,arnold2008finite, hu2015family,hu2016finite,hu2014finite,HongHuMaXu2021} and also the biharmonic equation~\cite{HuMaZhang2021}. 
The construction in~\cite{HuMaZhang2021} for two and three dimensions has been generalized to finite elements for $H(\rm div \, \rm div)$- and $H(\rm div)$-conforming symmetric tensors on simplexes in arbitrary dimension in~\cite{ChenHuang2022} using characterizations of degrees of freedom (DOF) in the dual spaces of trace and bubble spaces and certain subspaces
of the latter. 

As the $H(\rm div \, \rm div)$-conformity of the symmetric tensor-valued polynomial stress approximation $\sig_h$ is achieved by enforcing the normal continuity of $\div \sig_h$, this inheritance makes the basis functions easy to compute and imposes extra regularity. As shown in~\cite{ChenHuang2022}, a simple modification of the DOF can be used in order to reduce smoothness and construct symmetric tensor approximations that are only $H(\rm div \, \rm div)$-conforming but no longer $H({\rm div} \, {\rm div}) \cap H({\rm div})$-conforming. Moreover, redistributing the DOF to edges and faces has been demonstrated to be a proper means to accomplish a hybridizable mixed method with super-convergence for the biharmonic equation in~\cite{ChenHuang2025}.

\subsection*{Main contributions}
The main contributions of the present work are as follows: First, we introduce a new mixed-mixed formulation
based on a novel (nonstandard) Hilbert space for the quasi-static Biot problem, or, more specifically, the static
continuous in space problem resulting from it by implicit time stepping, for instance, using the backward Euler
scheme. { The proposed variational formulation relies directly on the fundamental physics principles,
namely, the momentum balance, the mass balance, Darcy's law and the generalised Hooke's law for linear elasticity
in compliance form. We neither introduce any new variables like the total stress or a total pressure, nor 
do we rewrite any equations because we aim at computing the original physical quantities of interest directly.}
The resulting physics-driven variational problem is proven to be well-posed by using the novel Hilbert space. 

Next, we exploit the discrete spaces introduced in~\cite{ChenHuang2022}, here in order to define a family of conforming strongly conservative mixed-mixed finite element methods for the Biot problem in the considered four-field formulation. The arising finite-dimensional variational problem is proven to be well-posed and optimal error estimates for the approximations
of all four fields are derived. The estimates are optimal in $H(\rm div)$ norm for the stress approximation and both in energy norm related to the novel Hilbert space and in $L^2$ norm for the approximation of the displacement-flux pair, as well as in $L^2$ norm for the pressure approximation.
%Under sufficient smoothness of the pressure field, 
We further show that a cost-efficient post-processing procedure can be applied to lift the pressure approximation into a space of piecewise polynomials
of one degree higher increasing the pressure-approximation order by one.

The paper is organized as follows. Section~\ref{sec:formulation} discusses the strong formulation of the problem under investigation. Section~\ref{sec:weak_form} presents a novel variational four-field formulation, starting with a discussion
of the underlying function spaces in Section~\ref{subsec:function_spaces}, exposing the continuous variational problem in Section~\ref{subsec:variational_formulation}, and showing its well-posedness in Section~\ref{subsec:well_posedness}. Section~\ref{sec:fem} establishes theoretically a family of strongly conservative conforming finite element methods for Biot's model of consolidation by defining the underlying discrete spaces in Section~\ref{subsec:discret_spaces}, proposing the corresponding mixed-mixed method in Section~\ref{subsec:mixed_mixed_fem},
{
proving discrete well-posedness in Section~\ref{subsec:fem_stability_estimate}, and providing optimal a priori error estimates in Section~\ref{subsec:fem_error_estimate} as well as a post-processing procedure to increase the order of the pressure
approximation in~Section~\ref{subsec:fem_post_processing}.
%Finally, Section~\ref{sec:num_res} presents numerical results obtained with the proposed family
%of strongly conservative mixed-mixed methods that demonstrate the ease of implementation and
%confirm the theoretical error estimates.
}

\section{Problem formulation}\label{sec:formulation} 

In this work we consider a four-field formulation of Biot's quasi-static model of consolidation in a bounded Lipschitz domain
$\Omega$ in $\mathbb{R}^{d}, d=1,2,3$ modelling the flow of a viscous incompressible fluid through an elastic porous matrix 
that is deformed by the fluid pressure as well as surface and/or body forces acting on the solid. The four quantities of interest are collected in a solution vector $\bx := (\sig,(\bu,\bw),p)^T$ where $\sig:=\sig(\bu)$ denotes the effective (elastic) stress, $\bu$ the displacement field, $\bw$ the Darcy velocity of the fluid (flux), and $p$ the fluid pressure. Under the assumption of linear elasticity, the stress is proportional to the strain, expressed by Hooke's law
\begin{equation}\label{eq:Hook}
\sig = C \eps(\bu) = 2 \mu \eps(\bu) + \lambda \divv \bu I ,
\end{equation}
where $I$ denotes the identity tensor and $\eps(\bu) = \frac{1}{2}(\nabla \bu + (\nabla \bu)^T)$ the symmetric part of the gradient of
a weakly differentiable vector field $\bu$. Then, with the fourth-order compliance tensor $A$ defined by
$$
A \sig = \frac{1}{2} \mu \left(
\sig - \frac{\lambda}{d \lambda+2 \mu}
\text{tr}(\sig) I  \right)
$$
for given Lam\'{e} parameters $\lambda$ and $\mu$, we have the relation
\begin{equation}\label{eq:Hook_2}
A \sig = \eps(\bu).
\end{equation}
The mathematical model we study is described by the following system of partial differential equations
\begin{subequations}\label{eq:Biot_strong}
\begin{align}
A \sig - \eps(\bu) &= {\bm 0} \quad \text{in}~~ \Omega\times (0,T),
\label{Biot1_strong}\\
-\div \sig  +  \alpha \nabla p 
&= \bff \quad \text{in}~~ \Omega\times (0,T),
\label{Biot2_strong}\\
R_p^{-1}  \bw + \nabla p&= \boldsymbol 0 \quad \text{in}~~ \Omega\times (0,T), \label{Biot3_strong}\\
\partial_t \left( \alpha  \div \bu + S p\right) + \div \bw  &=g \quad \text{in}~~ \Omega\times (0,T).
\label{Biot4_strong}
\end{align}
\end{subequations}
%
%The initial and boundary conditions to the sought solution $\bx$ of the system~\eqref{eq:Biot_strong} are as follows.
We assume that at the initial time moment $t=0$ the fluid pressure $p$ and the displacement field $\bu$ satisfy the initial conditions
\begin{subequations}\label{eq:Biot_IC}
\begin{eqnarray}
\label{eq:Biot_IC_p}
p(x,0) &=& p_{0}(x)  \qquad \mbox{for } x \in \Omega,\\
\label{eq:Biot_IC_u}
\bu(x,0) &=& {\bu}_0(x) \qquad \mbox{for }  x \in \Omega .
\end{eqnarray}
\end{subequations}
Moreover, the following boundary conditions 
\begin{subequations}\label{eq:Biot_BC}
\begin{eqnarray}
\label{eq:Biot_BC_pD}
p(x,t) &=& p_{D}(x,t)  \quad \mbox{for } x \in \Gamma_{p,D}, \quad t > 0,\\
\label{eq:Biot_BC_pN}
R_p \displaystyle \frac{\partial p(x,t)}{\partial \bn} &=& p_{N}(x,t)  \quad \mbox{for }  
x \in \Gamma_{p,N}, \quad t > 0,\\ 
\label{eq:Biot_BC_uD}
\bu(x,t) &=& {\bu}_D(x,t)  \quad \mbox{for }  x \in \Gamma_{\bu,D}, \,\quad t > 0, \\ 
\label{eq:Biot_BC_uN}
%(\sig(x,t)-\alpha p  \boldsymbol I) \, {\bn} (x) &=
\sig(x,t)\, {\bn} (x) &=
& {\bu}_N(x,t) \quad \mbox{for }  x \in \Gamma_{\bu,N}, \,\quad t > 0, ~~~~~~~~~~~~~~~~~~~~
\end{eqnarray}
\end{subequations}
are imposed on $p$ and $\bu$ where
$\Gamma_{p,D} \cap \Gamma_{p,N} = \emptyset$,
$\overline{\Gamma}_{p,D}\cup \overline{\Gamma}_{p,N}=\Gamma=\partial{\Omega}$
and
$\Gamma_{\bu,D} \cap \Gamma_{\bu,N} = \emptyset$,
$\overline{\Gamma}_{\bu,D} \cup \overline{\Gamma}_{\bu,N}=\Gamma$. 

After time discretization of~\eqref{eq:Biot_strong} by an implicit (or semi-implicit) time integration method
and a variable substitution, similar to what has been shown in~\cite{HongKraus2018parameter}, the static problems to be solved in each time step, written as operator equations,
are of the form
\begin{equation}\label{eq:operator_form}
\bm{A} \bm x = \bm{F},
\end{equation}
for a solution vector $\tilde \bx := (\sig,(\tilde \bu,\tilde \bw),\tilde p)^T$ and right had side
$\bm{F}=(\bm 0,(- \tilde \bff, \bm 0), \tilde{g})^T$, where the operator matrix $\bm{A}$ is defined by
\begin{equation}\label{eq:operator_A}
\bm{A} := \left[\begin{array}{cccc}
A & -\eps & 0 & 0  \\
\div & 0 & 0 & - \nabla  \\
0 & 0 & -R_p^{-1} & - \nabla   \\
0 & \div & \div & S I
\end{array}\right] ,
\end{equation}
and $\tilde{g}$ depends on the specific time integrator; The symbols with tilde indicate that the quantities
they denote are related to the corresponding (original) quantities by a proper scaling with problem
parameters and/or time step size.\footnote{Note that for ease of notation, we will skip the tilde symbol again 
in what follows, being aware of the fact that the recovery of the original solution vector requires the appropriate
rescaling.}  

{
Note that problem~\eqref{eq:operator_form}--\eqref{eq:operator_A}, apart from rescaling, which is a matter of convenience, directly reflects the physics laws of Biot's model.}

\section{Weak form of the problem}\label{sec:weak_form} 

In this section, we present the weak formulation of the (static) four-field model~\eqref{eq:operator_form}--\eqref{eq:operator_A}
that arises in each step of an implicit (or semi-implicit) time-stepping method.

\subsection{Function spaces}\label{subsec:function_spaces} 

To begin with, we introduce some notation thereby also recalling certain standard Hilbert spaces.

Let $\mathbb{V}$ denote the set of all real $d$ vectors, $\mathbb{M}$ the set of all real $d{\times}d$ matrices, and $\mathbb{S}$ its subset of
symmetric $d{\times}d$ matrices. 
Further, with $L^2(\Omega;\mathbb{V})$ and $L^2(\Omega;\mathbb{S})$ we denote the Hilbert spaces of all real vector-valued and symmetric tensor-valued functions whose entries are in $L^2(\Omega)$, i.e., square Lebesgue-integrable scalar functions, respectively.

Next, we define the Hilbert space
$$
\Sigma := H(\div,\Omega;\mathbb{S}) :=\{ \sig  \in L^2(\Omega;\mathbb{S}): \div \sig \in  L^2(\Omega;\mathbb{V}) \},
$$
its subspace
$$
%\Sigma^0 :=\{ \sig \in \Sigma: (\sig -\alpha p  \boldsymbol I) \, \bn = \bm 0 \text{ on }\Gamma_{\bu,N}\},
\Sigma^0 :=\{ \sig \in \Sigma: \sig \, \bn = \bm 0 \text{ on }\Gamma_{\bu,N}\},
$$
and the set
$$
%\Sigma^D :=\{\sig \in \Sigma: (\sig -\alpha p  \boldsymbol I) \, \bn = \bu_{N}(x,t) \text{ on }\Gamma_{\bu,N}\},
\Sigma^D :=\{\sig \in \Sigma: \sig  \, \bn = \bu_{N}(x,t) \text{ on }\Gamma_{\bu,N}\},
$$
equipped with the norm $\Vert \cdot \Vert_\Sigma$ defined by
$$
\Vert \sig \Vert_\Sigma^2 := \Vert \sig \Vert^2 + \Vert \div \sig \Vert^2
$$
where $\Vert \cdot \Vert$ denotes the $L^2$ norm for functions in $L^2(\Omega;\mathbb{S})$ or $L^2(\Omega;\mathbb{V})$ (or $L^2(\Omega)$).

For a pair of vector fields $\bar{\bv} := (\bu,\bw)^T$ we introduce the novel Hilbert space
$$
\bar{V} := \{ \bar{\bv} = (\bu,\bw)^T:  \bu \in L^2(\Omega;\mathbb{V}), \bw \in L^2(\Omega;\mathbb{V}), \mbox{ and } \div (\bu + \bw) \in  L^2(\Omega) \},
$$
as well as its subspace
$$
\bar{V}^0 := \{ \bar{\bv} = (\bu,\bw)^T  \in \bar{V}: \bw \cdot \bn = 0 \text{ on }\Gamma_{p,N} \}
$$
both equipped with the norm $\Vert \cdot \Vert_{\bar{V}}$ defined by
$$
\Vert \bar{\bv} \Vert_{\bar{V}}^2 := \Vert (\bu,\bw)^T  \Vert_{\bar{V}}^2 
:= \Vert \bu \Vert^2 + \Vert \bw \Vert^2 + \Vert \div (\bu + \bw)  \Vert^2.
$$

Moreover, we will make use of the Hilbert space
$$
H(\div \div,\Omega;\mathbb{S}) := \{ \sig  \in L^2(\Omega;\mathbb{S}): \div \div \sig \in  L^2(\Omega) \}
$$
equipped with the norm $\Vert \cdot \Vert_{H(\div \div,\Omega;\mathbb{S}) }$ defined by
$$
\Vert \sig \Vert_{H(\div \div,\Omega;\mathbb{S})}^2 := \Vert \sig \Vert^2 + \Vert \div \div \sig \Vert^2.
$$
Later on, in Section~\ref{sec:fem}, we will also refer to the intersection $H(\div \div,\Omega;\mathbb{S}) \cap \Sigma$ of these two Hilbert spaces.

\subsection{Variational formulation}\label{subsec:variational_formulation} 

The weak form of problem~\eqref{eq:operator_form} reads. Find
$\bx := (\sig,(\bu,\bw),p)^T \in X:= \Sigma^0 \times \bar{V}^0 \times L^2(\Omega)$ which satisfies the equations
\begin{subequations}\label{eq:Biot_weak}
\begin{align}   
(A \sig, \bm \tau) + (\bu,\divv \bm \tau) = &  \langle \bu_D, 
\bm \tau \bn \rangle_{\Gamma_{\bu,D}}, \label{eq:Biot_weak1}\\
(\div \bm \sigma,\bm v) + (p, \div \bv) = & \; -(\bff,\bv),  \label{eq:Biot_weak2}\\  
- (R_p^{-1}  \bw, \bz) + (p,\div \bz) = &
\langle p_{D}, \bz \cdot \bn \rangle_{\Gamma_{p,D}}, \label{eq:Biot_weak3} \\ 
(\div \bu, q) + (\div \bw, q) + (S p,q) = & (\tilde{g},q) \label{eq:Biot_weak4} 
\end{align}
\end{subequations}
for all $\by := (\bm \tau,(\bv,\bz),q)^T \in X$.
Note that for simplicity we have assumed that $p_{N}(x,t) = 0$ and ${\bu}_N(x,t)= \bm 0$, which in
presence of inhomogeneous essential boundary conditions can be achieved by homogenization.

\begin{remark} \label{Biot_Stokes}
Note that the mixed-mixed  weak formulation of the Stokes problem can be viewed as a special case of problem \eqref{eq:Biot_weak} where the filed $\bw$
and equation \eqref{eq:Biot_weak3} are not present and $A$ equals the identity tensor in \eqref{eq:Biot_weak1}. In this case, equation \eqref{eq:Biot_weak4} reduces to $(\div \bu, q)=0$ and the Hilbert space $\bar{V}$ reduces to the $H(\div,\Omega; \mathbb{V})$ space.
\end{remark}

\subsection{Well-posedness analysis}\label{subsec:well_posedness} 

For the well-posedness analysis of the continuous problem~\eqref{eq:Biot_weak} we assume only that 
$\sig \in \Sigma^0$ or $\sig \in \Sigma$, the latter if $\Gamma_{\bu,N} = \emptyset$. 

Associated with the weak formulation~\eqref{eq:Biot_weak} we consider the bilinear form
\begin{align}\label{eq:Biot_bilinear_form}  
\mathcal A ((\sig,(\bu,\bw),p),(\bm \tau,(\bv,\bz),q))
&:=(A \sig, \bm \tau) + (\bu,\div \bm \tau) 
+ (\div \bm \sigma,\bm v) - (R_p^{-1}  \bw, \bz) \nonumber \\
& + (p,\div (\bv + \bz)) + (\div (\bu + \bw),q) + (S p,q)
\end{align}
on $X {\times} X$ and the linear form
\begin{align}\label{eq:Biot_linear_form}  
\mathcal F ((\bm \tau,(\bv,\bz),q))
&:=  \langle \bu_D, 
\bm \tau \bn \rangle_{\Gamma_{\bu,D}}
- (\bff,\bv)
+ \langle p_{D}, \bz \cdot \bn \rangle_{\Gamma_{p,D}}
+ (\tilde{g},q)
\end{align}
on $X$ with which problem~\eqref{eq:Biot_weak} equivalently can be expressed in the form:
Find $\bx := (\sig,(\bu,\bw),p)^T \in X:= \Sigma^0 \times \bar{V}^0 \times L^2(\Omega)$ which satisfies
\begin{equation}\label{eq:Biot_weak_abstract}
\mathcal A ((\sig,(\bu,\bw),p),(\bm \tau,(\bv,\bz),q))
=\mathcal F ((\bm \tau,(\bv,\bz),q)) \quad \forall \by := (\bm \tau,(\bv,\bz),q)^T \in X.
\end{equation}
Defining now the (combined) norm $\Vert \cdot \Vert_X$ by
\begin{equation}\label{eq:Xnorm}
\Vert (\bm \tau,(\bv,\bz),q)^T \Vert_X^2 := \Vert \bm \tau \Vert_\Sigma^2 + \Vert (\bv,\bz)^T \Vert_{\bar{V}}^2 + \Vert q\Vert^2 ,
\end{equation}
a simple calculation using Cauchy-Schwarz inequality shows that $\exists C_0 > 0$, in general depending on certain problem
parameters, such that
\begin{equation}\label{eq:boundedness}
\mathcal A ((\sig,(\bu,\bw),p),(\bm \tau,(\bv,\bz),q)) \le  C_0 \Vert (\sig,(\bu,\bw),p)^T \Vert_X \Vert (\bm \tau,(\bv,\bz),q)^T \Vert_X
%\quad \forall (\sig,(\bu,\bw),p)^T, (\bm \tau,(\bv,\bz),q)^T \in X.
\end{equation} 
for all $((\sig,(\bu,\bw),p)^T, (\bm \tau,(\bv,\bz),q)^T) \in X \times X$.
Moreover, we take note of the following two stability estimates:
\begin{equation}\label{eq:inf_sup_sigma_v}
\exists \beta_1 > 0: \quad \inf_{\substack{\bv \in L^2(\Omega;\mathbb{V}) \\ \bv \neq \bm 0}} \sup_{\substack{\sig \in \Sigma \\ \sig \neq \bm 0}} 
\frac{(\div \sig,\bv)}{\Vert \sig \Vert_\Sigma \Vert \bv \Vert} \ge \beta_1 > 0 ,
\end{equation} 
\begin{equation}\label{eq:inf_sup_w_q}
\exists \beta_2 > 0: \quad \inf_{\substack{q \in L^2(\Omega) \\ q \neq 0}} \sup_{\substack{\bm w \in H(\div,\Omega; \mathbb{V}) \\ \bm w \neq \bm 0}} 
\frac{(\div \bm w,q)}{\Vert \div \bm w \Vert_{H(\div,\Omega; \mathbb{V})} \Vert q \Vert} \ge \beta_2 > 0 .
\end{equation} 
With these two conditions at hand, we are able to establish the well-posedness of problem~\eqref{eq:Biot_weak_abstract}, or, equivalently problem~\eqref{eq:Biot_weak}. 
\begin{theorem}\label{thm:well-posedness}
Problem~\eqref{eq:Biot_weak_abstract} with $\mathcal A (\cdot,\cdot)$ defined by~\eqref{eq:Biot_bilinear_form}
and $\mathcal F (\cdot)$ defined by~\eqref{eq:Biot_linear_form} is well-posed.
%i.e., it has a unique solution $(\sig,(\bu,\bw),p)^T \in X$. 
In particular, the continuity estimate~\eqref{eq:boundedness} and the inf-sup condition 
%\begin{align}
%\mathcal A ((\sig,(\bu,\bw), p),(\tau,(\bv,\bz),q))\le C_0  \left(\Vert \sig \Vert_\Sigma + \Vert (\bu,\bw)^T \Vert_{\bar{V}} + \Vert p \Vert\right) (\Vert \btau \Vert_\Sigma + \Vert (\bv,\bz)^T \Vert_{\bar{V}} + \Vert q \Vert)
%\end{align}
%{ for all $(\sig,(\bu,\bw), p)^T, (\btau,(\bv,\bz),q)^T \in X$ and}
\begin{align}
\sup_{\btau \in \Sigma, (\bv,\bz)\in \bar{V}, q\in P}\frac{\mathcal A ((\sig,(\bu,\bw),p),(\btau,(\bv,\bz),q))}{\Vert \btau \Vert_\Sigma + \Vert (\bv,\bz)^T \Vert_{\bar{V}} + \Vert q \Vert}\ge c_0\left(\Vert \sig \Vert_\Sigma + \Vert (\bu,\bw)^T \Vert_{\bar{V}} + \Vert p \Vert\right),
\end{align}
for all  $(\sig,(\bu,\bw),p)^T \in X$,
hold true with constants $C_0$ and $c_0$ that do not depend on the parameter~$\lambda$. 
\end{theorem}
\begin{proof}
In view of Babu\v{s}ka's theorem~\cite{Babuska1971error} { it suffices to verify the following three conditions. Firstly, the linear form~\eqref{eq:Biot_linear_form} in problem~\eqref{eq:Biot_weak_abstract} is bounded in $X$-norm, which is obvious from the definitions. Secondly,  the bilinear form~\eqref{eq:Biot_bilinear_form} is bounded in $X$-norm, i.e., the continuity condition~\eqref{eq:boundedness} holds, and, thirdly, it is also inf-sup-stable, i.e.,  $\mathcal A (\cdot,\cdot)$ satisfies the condition}
\begin{equation}\label{eq:inf_sup_A}
\exists c_0 > 0: \inf_{\bx \in X} \sup_{\by \in X} \frac{\mathcal{A}(\bx^T,\by^T)}{\Vert \bx \Vert_{X} \Vert \by \Vert_{X}} \ge c_0 > 0
\end{equation}
where $\bx = (\sig,(\bu,\bw),p)^T, \by = (\bm \tau,(\bv,\bz),q)^T$, and $X= \Sigma^0 \times \bar{V}^0 \times L^2(\Omega)$.
We restrict ourselves here to show~\eqref{eq:inf_sup_A} because the verification of the other assumptions of the theorem is straightforward.

Due to the symmetry of $\mathcal A (\cdot,\cdot)$ we have
$$
\inf_{\bx \in X} \sup_{\by \in X} \frac{\mathcal{A}(\bx^T,\by^T)}{\Vert \bx \Vert_{X} \Vert \by \Vert_{X}}
= \inf_{\by \in X} \sup_{\bx \in X} \frac{\mathcal{A}(\bx^T,\by^T)}{\Vert \bx \Vert_{X} \Vert \by \Vert_{X}} .
$$
{
If we are able to show that for an arbitrary element $\by = (\bm \tau,(\bv,\bz),q)^T \in X$ we can find
an element $\tbx_0 = (\tsig,(\tbu,\tbw), \tp)^T \in X$ such that
\begin{equation}\label{eq:boundedness_estimate}
\Vert \tbx_0  \Vert_X^2 =
\Vert \tsig \Vert_\Sigma^2 + \Vert (\tbu,\tbw)^T \Vert_{\bar{V}}^2 + \Vert \tp \Vert^2 
\le c_1 (\Vert \bm \tau \Vert_\Sigma^2 + \Vert (\bv,\bz)^T \Vert_{\bar{V}}^2 + \Vert q \Vert^2)
= c_1 \Vert \by  \Vert_X^2
\end{equation}
and
\begin{equation}\label{eq:coercivity_estimate}
\mathcal A (\tbx_0^T,\by^T) 
=
\mathcal A ((\tsig,(\tbu,\tbw),\tp),(\bm \tau,(\bv,\bz),q)) 
\ge c_2 (\Vert \bm \tau \Vert_\Sigma^2 + \Vert (\bv,\bz)^T \Vert_{\bar{V}}^2 + \Vert q \Vert^2) 
= c_2 \Vert \by  \Vert_X^2
\end{equation}
hold true for positive constants $c_{1}$ and $c_{2}$, the inf-sup condition~\eqref{eq:inf_sup_A} follows
from~\eqref{eq:boundedness_estimate} and~\eqref{eq:coercivity_estimate} since
\begin{align}\label{eq:continuous_inf_sup}
\sup_{\bx \in X} \frac{\mathcal{A}(\bx^T,\by^T)}{\Vert \bx \Vert_{X}}
& \ge  \frac{\mathcal{A}(\tbx_0^T,\by^T)}{\Vert \tbx_0 \Vert_{X}}
\ge \frac{c_2 \Vert \by \Vert_{X}^2}{\Vert \tbx_0 \Vert_{X}}
\ge \frac{c_2}{\sqrt{c_1}} \Vert \by \Vert_{X} , \mbox{ for all } \by \in X .
%= (\tau_h,(\bv_h,\bz_h),q_h)^T \in X_h
\end{align}
}

So we consider an arbitrary but fixed element $\by = (\bm \tau,(\bv,\bz),q)^T \in X$.
Then, in view of the inf-sup condition~\eqref{eq:inf_sup_sigma_v} there exists an element $\sig_0 \in \Sigma$ such that
\begin{equation}\label{eq:def_sig_0}
\div \sig_0 = \bv \mbox{ and } \Vert \sig_0 \Vert_\Sigma \le \beta_1^{-1} \Vert \bv \Vert .
\end{equation}
Moreover, due to the inf-sup condition~\eqref{eq:inf_sup_w_q} there exists an element $\bw_0 \in H(\div,\Omega;\mathbb{V})$ such that
\begin{equation}\label{eq:def_w_0}
\div \bw_0 = q \mbox{ and } \Vert \bw_0 \Vert_{H(\div,\Omega;\mathbb{V})} \le \beta_2^{-1} \Vert q \Vert .
\end{equation}
We then define $\tbx := (\tsig,(\tbu,\tbw), \tp)^T \in X$ in terms of
\begin{subequations}\label{eq:def_x_tilde}
\begin{align}   
\tsig & := \sig_0 + \delta_1 \btau , \label{eq:def_sig_tilde} \\
\tbu & := - \delta_1 \bv + \div \btau , \label{eq:def_u_tilde}  \\
\tbw & := \bw_0 - \div \btau - \delta_1 \bz , \label{eq:def_w_tilde}  \\
%%(\tbu,\tbw)
\tp & : = \delta_1 q + \div ( \bv + \bz ) \label{eq:def_p_tilde}
\end{align}
\end{subequations}
for a positive constant $\delta_1 < \infty$, which we will determine later,
and find that
\begin{align*} 
\Vert \tsig \Vert_\Sigma^2 + \Vert (\tbu,\tbw)^T \Vert_{\bar{V}}^2 + \Vert \tp \Vert^2 
& \le \Vert \tsig \Vert^2 + \Vert \div \tsig \Vert^2 + \Vert \tbu \Vert^2 + \Vert \tbw \Vert^2 + \Vert \div (\tbu + \tbw) \Vert^2 + \Vert \tp \Vert^2 \\
& \le \tilde{c} \max (1,\delta_1^2) ( \Vert \sig_0 \Vert^2 +  \Vert \btau \Vert^2  + \Vert \div \sig_0 \Vert^2  + \Vert \div \btau \Vert^2  +  \Vert \bv \Vert^2  \\
& \ \ \ + \Vert \bw_0 \Vert^2 + \Vert \div \bw_0 \Vert^2 + \Vert \bz \Vert^2 + \Vert q \Vert^2 + \Vert \div ( \bv + \bz ) \Vert^2 ) \\
& \le c_1 (\Vert \bm \tau \Vert_\Sigma^2 + \Vert (\bv,\bz)^T \Vert_{\bar{V}}^2 + \Vert q \Vert^2)
\end{align*}
where $c_1$ depends only on $\delta_1$.

On the other hand, for the above choice~\eqref{eq:def_x_tilde} of $\tbx = (\tsig,(\tbu,\tbw), \tp)^T$ we have
\begin{align*}
\mathcal A ((\tsig,(\tbu,\tbw),\tp),(\btau,(\bv,\bz),q)) 
& = (A (\sig_0 + \delta_1 \btau) , \btau )  + (- \delta_1 \bv + \div \btau , \div \btau ) + ( \div (\sig_0 + \delta_1 \btau) , \bv ) \\
& - (R_p^{-1} (\bw_0 - \div \btau - \delta_1 \bz) , \bz)  + (\delta_1 q + \div ( \bv + \bz ) , \div ( \bv + \bz ) ) \\
& + ( \div (- \delta_1 \bv + \bw_0 - \delta_1 \bz) , q ) + (S (\delta_1 q + \div ( \bv + \bz ) ) , q) \\
& = (A \sig_0, \btau ) +  \delta_1 (A \btau, \btau ) +  (\div \btau , \div \btau ) + (\bv, \bv) \\
& - (R_p^{-1} \bw_0, \bz) + (R_p^{-1} \div \btau , \bz)  +  \delta_1 ( R_p^{-1} \bz , \bz ) \\
& + (\div ( \bv + \bz ) , \div ( \bv + \bz ) )  + (q,q) + \delta_1 (S q, q)  + (S \div ( \bv + \bz ), q) ,
\end{align*}
which, using Young's inequality, can further be estimated from below by
\begin{align*}
\mathcal A ((\tsig,(\tbu,\tbw),\tp),(\btau,(\bv,\bz),q)) 
& \ge  \delta_1 (A \btau, \btau ) - \frac{\epsilon_1}{2}  (A \sig_0, \sig_0) - \frac{1}{2 \epsilon_1} (A \btau, \btau ) \\
& + (\bv,\bv) + \delta_1 ( R_p^{-1} \bz , \bz ) +  (\div \btau , \div \btau ) \\
& -  \frac{\epsilon_2}{2}  (R_p^{-1}  \bw_0, \bw_0) - \frac{1}{2 \epsilon_2} (R_p^{-1}  \bz, \bz) \\
& - \frac{\epsilon_3}{2}  (R_p^{-1} \div \btau, \div \btau) - \frac{1}{2 \epsilon_3} (R_p^{-1}  \bz, \bz) \\
& + (q,q) + \delta_1 ( S q, q) +  (\div ( \bv + \bz ) , \div ( \bv + \bz ) ) \\
& - \frac{\epsilon_4}{2}  (S \div ( \bv + \bz ), \div ( \bv + \bz )) - \frac{1}{2 \epsilon_3} (S q, q) \\
& \ge c_2 (\Vert \bm \tau \Vert_\Sigma^2 + \Vert (\bv,\bz)^T \Vert_{\bar{V}}^2 + \Vert q \Vert^2)
\end{align*}
for some positive constant $c_2$, provided the constants $\epsilon_i$, $i \in \{1,2,3,4 \}$ and $\delta_1$ are chosen properly,
($\epsilon_i$ small enough, depending on the model parameters and $\delta_1$ large enough, depending on the $\epsilon_i$).
Here we have also used the fact that $ \Vert \sig_0 \Vert_\Sigma \le \beta_1^{-1} \Vert \bv \Vert$ and
$ \Vert \bw_0 \Vert_{H(\div,\Omega;\mathbb{V})} \le \beta_2^{-1} \Vert q \Vert$.
This completes the proof.
\end{proof}
\begin{remark}\label{regularity}
By choosing $\btau\in C_0^{\infty}(\Omega;\mathbb{S})$ and using the definition of weak derivatives, we see that the solution $\bu$ of problem \eqref{eq:Biot_weak_abstract} satisfies $\bu\in H^1(\Omega,\mathbb{V})$, where 
$$
H^1(\Omega; \mathbb{V}):=\{\bu\in L^2(\Omega,\mathbb{V}): ~\nabla \bu \in L^2(\Omega,\mathbb{M})\}.
$$
 Therefore, $\bar{\bu}=(\bu,\bw)^T\in \bar{V}$ implies $\div (\bu+\bw)\in L^2(\Omega)$ and $\div \bw=\div (\bu+\bw)-\div\bu\in L^2(\Omega)$, which means $\bw\in H(\div,\Omega;\mathbb{V})$, where
$$
H(\div,\Omega;\mathbb{V}):=\{\bw \in L^2(\Omega,\mathbb{V}): ~\div \bw \in L^2(\Omega)\}.
$$
\end{remark}

\section{Family of strongly conservative finite element discretizations}\label{sec:fem} 
Following the works of J.~Hu, R.~ Ma, and M.~Zhang~\cite{HuMaZhang2021}, and L.~Chen and X.~Huang~\cite{ChenHuang2022}, the goal of this section is to describe a family of strongly
conservative finite element methods for the discretization of problem~\eqref{eq:Biot_weak_abstract}.
For this purpose we use the $H(\div \div) \cap H(\div)$-conforming element originally proposed in ~\cite{HuMaZhang2021}
and ~\cite{ChenHuang2022} here as a component of a pointwise, i.e., angular momentum- and mass-conserving,
mixed-mixed method for the weak four-field formulation of the the quasi-static Biot problem exposed in Section~\ref{sec:weak_form}.

\subsection{Discrete spaces}\label{subsec:discret_spaces} 
$H(\div\div)$-conforming finite elements for symmetric tensors are presented in~\cite{chen2020finite, HuMaZhang2021, chen2022finite} in two and three space dimensions, $H(\div \div)$-conforming finite elements for symmetric tensors in arbitrary dimension have recently been proposed in~\cite{ChenHuang2022}. 
Moreover, in their subsequent work~\cite{ChenHuang2025}, the authors of  ~\cite{ChenHuang2022} removed the need for super-smoothness
by redistributing the degrees of freedom of these elements to edges and faces. An a posteriori error analysis for the
symmetric $H(\div\div)$ mixed finite element method for the Kirchhoff-Love plate bending problem is conducted in \cite{hu2025optimality}. 

With the aim of keeping this presentation as self-contained as possible, we recall in detail the following degrees of freedom (DOF) as an example for the $H(\div\div)$-conforming  
$\P_k(K ; \mathbb{S})$ element  with $k\geq \max\{ d,3\}$ from~\cite{ChenHuang2022}:
\begin{align}
\boldsymbol{\tau}(\delta) & \quad \forall \delta \in \mathcal{V}(K), \label{DOF0}\\
\left(\boldsymbol{n}_i^{\top} \boldsymbol{\tau} \boldsymbol{n}_j, q\right)_f &\quad \forall q \in \P_{k+s-d-1}(f),
\quad f \in \mathcal{F}^s(K), \quad s=1, \ldots, d-1, \quad i, j \in {1, \ldots, s}, \label{DOF1}  \\
\left(\Pi_F \boldsymbol{\tau} \boldsymbol{n}, \boldsymbol{q}\right)_F &\quad \forall \boldsymbol{q} \in \mathrm{ND}_{k-2}(F), \quad F \in \mathcal{F}^1(K), \label{DOF2}  \\
\left(\boldsymbol{n}^{\top} \operatorname{div} \boldsymbol{\tau}, q\right)_F & \quad \forall q \in \P_{k-1}(F), 
\quad F \in \mathcal{F}^1(K), \label{DOF3} \\
(\div \div \boldsymbol{\tau}, q)_K & \quad \forall q \in \P_{k-2}(K) / \P_1(K),  \label{DOF4}\\
(\div \boldsymbol{\tau}, \boldsymbol{q})_K & \quad \forall \boldsymbol{q} \in\left(\P_{k-3}(K ; \mathbb{K}) / \P_0(K ; \mathbb{K})\right) x, \label{DOF5}\\
(\boldsymbol{\tau}, \boldsymbol{q})_K & \quad \forall \boldsymbol{q} \in \operatorname{ker}(\cdot x) \cap \P_{k-2}(K ; \mathbb{S}), \label{DOF6}
\end{align}
where $\mathcal{V}(K)$ denotes the set of all vertices of a $d$-dimensional simplex $K$, $\mathcal{F}^s(K)$ denotes the set of subsimplexes of $K$ with co-dimension $s$ for $s=1, \ldots, d-1$,  and $\Pi_F\boldsymbol{\tau}$ denotes the projection of column vectors of $\boldsymbol{\tau}$ to the hyperplane $F$. Moreover,  $\mathrm{ND}_{k-2}(F)$ denotes the shape function space of the first kind of Ned\'el\'ec edge element on $F$, see~\cite{nedelec1980mixed, arnold2006finite}.
The set of scalar-valued polynomials of degree $\ell$ on $K$ is denoted by $\P_{\ell}(K)$, the set of skew-symmetric tensor-valued polynomials of degree $\ell$ is denoted by $\P_{\ell}(K ; \mathbb{K})$, and the set of symmetric tensor-valued polynomials of degree $\ell$ by $\P_{\ell}(K ; \mathbb{S})$. Furthermore, $\operatorname{ker}(\cdot x) \cap \P_{\ell}(K ; \mathbb{S})=x^\perp (x^\perp)^T\P_{\ell-2}(K)$ for $d=2$ with $x^\perp=(x_2, -x_1)^T$ and $\operatorname{ker}(\cdot x) \cap \P_{\ell}(K ; \mathbb{S})=x^T \times \P_{\ell-2}(K ; \mathbb{S})\times  x^T$ for $d=3$, see~\cite{chen2020finite,chen2022finiteele}.

These DOF enforce that $\operatorname{div} \boldsymbol{\tau}$ is $H(\operatorname{div})$-conforming, thus $\boldsymbol{\tau} \in H(\div \div,\Omega;\mathbb{S}) \cap H(\div,\Omega;\mathbb{S})$, and define an $H(\div \div) \cap H(\div)$-conforming finite element. The corresponding global finite element space
\begin{equation}\label{eq:conformity_Sigma_h}
\Sigma_h \subset H(\div \div, \Omega ; \mathbb{S}) \cap H(\div, \Omega ; \mathbb{S})
\end{equation}
is given by
{
\begin{equation}\label{eq:Sigma_h}
\Sigma_h :=\{\boldsymbol{\tau} \in L^2(\Omega ; \mathbb{S}): \boldsymbol{\tau}\vert_K \in \P_k(K, \mathbb{S}) \quad \forall K \in \mathcal{T}_h \mbox{ and the DOF~\eqref{DOF0}-\eqref{DOF3} are single-valued} \} .
\end{equation}
}
Then, by construction, as shown in~\cite{ChenHuang2022}, the space $\Sigma_h$ defined in~\eqref{eq:Sigma_h} satisfies the important relation
\begin{equation}\label{eq:div_Sigma_h}
\div \Sigma_h = \text{BDM}_{k-1}
\end{equation}
where $\text{BDM}_{k-1}$ denotes the classical Brezzi-Douglas-Marini finite-element space of order $k-1$, cf.~\cite{Boffi2013mixed}.
Moreover, assuming $k \ge 3$, we will employ the finite element spaces
\begin{subequations}\label{eq:V_h_P_h}
\begin{align}
V_h & :=  \text{BDM}_{k-1} \label{eq:V_h} \\
\bar{V}_h & := \{ \bar{\bv}_h = (\bu_h,\bw_h)^T:  \bu_h , \bw_h \in V_h \}, \label{eq:barV_h} \\
P_h & := \{ q \in L^2(\Omega): q\vert_K \in \P_{k-2}(K)) \label{eq:P_h}
\end{align}
\end{subequations}
in the formulation of a conservative, i.e., mass- and angular momentum-conserving, discrete four-field Biot model that is stated in the following. 
\begin{remark}
Note that the space~\eqref{eq:barV_h} has built in a higher smoothness as compared to the continuous counterpart $\bar{V}$ for whose elements $(\bu,\bw)$ we only require $\div (\bu + \bw) \in  L^2(\Omega)$ and not the stronger condition that $\div \bu \in  L^2(\Omega)$ and $\div \bw \in  L^2(\Omega)$.  Functions in the space $\bar{V}_h$,
however, belong to $H(\div)$ componentwise due to~\eqref{eq:V_h} and~\eqref{eq:barV_h}.
\end{remark}

\subsection{Mixed-mixed finite element method}\label{subsec:mixed_mixed_fem}

To discretize the the four-field Biot model under investigation, we transfer the continuous variational problem~\eqref{eq:Biot_weak} to the finite element spaces
$$
\Sigma_h^0 := \{ \sig_h \in \Sigma_h: \sig_h \bn = \bm 0 \text{ on }\Gamma_{\bu,N}\},
$$
$$
\bar{V}_h^0 := \{ \bar{\bv_h} = (\bu_h,\bw_h)^T  \in \bar{V_h}: \bw_h \cdot \bn = 0 \text{ on }\Gamma_{p,N} \},
$$
and $P_h$, which are introduced via~\eqref{eq:Sigma_h} and~\eqref{eq:V_h_P_h}.

With
\begin{equation}\label{eq:X_h}
X_h = \Sigma_h^0 \times \bar{V}_h^0 \times P_h
\end{equation}
the conforming mixed-mixed finite element method reads: Find
$\bx_h := (\sig_h,(\bu_h,\bw_h),p_h)^T \in X_h$ satisfying the equations
\begin{subequations}\label{eq:Biot_fem}
\begin{align}   
(A \sig_h, \bm \tau_h) + (\bu_h,\div \tau_h) = &  \langle \bu_D, 
\bm \tau_h \bn \rangle_{\Gamma_{\bu,D}}, \label{eq:Biot_fem1}\\
(\div \sig_h,\bv_h) + (p_h, \div \bv_h) = & \; -(\bff,\bv_h),  \label{eq:Biot_fem2}\\  
- (R_p^{-1}  \bw_h, \bz_h) + (p_h,\div \bz_h) = &
\langle p_{D}, \bz_h \cdot \bn \rangle_{\Gamma_{p,D}}, \label{eq:Biot_fem3} \\ 
(\div \bu_h, q_h) + (\div \bw_h, q_h) + (S p_h,q_h) = & (\tilde{g},q_h) \label{eq:Biot_fem4} 
\end{align}
\end{subequations}
for all $\by_h := (\bm \tau_h,(\bv_h,\bz_h),q_h)^T \in X_h$.

\subsection{Discrete well-posedness}\label{subsec:fem_stability_estimate} 

In this subsection we want to prove that the discrete problem~\eqref{eq:Biot_fem} is also well-posed.
Our prove basically repeats the steps of the proof of Theorem~\ref{thm:well-posedness} and relies on
two discrete inf-sup-conditions. The first one is provided in the following lemma.
\begin{theorem}\label{lem:inf_sup}
There exists a constant $\beta_3$ which does not depend on the mesh size $h$ such that
\begin{equation}\label{eq:inf_sup_sigma_h_v_h}
\exists \beta_{3} > 0: \quad \inf_{\substack{\bv_h \in V_h \\ \bv_h \neq \bm 0}} \sup_{\substack{\sig_h \in \Sigma_h \\ \sig_h \neq \bm 0}}
\frac{(\div \sig_h,\bv_h)}{\Vert \sig_h \Vert_\Sigma \Vert \bv_h \Vert} \ge \beta_{3} > 0. 
\end{equation} 
\end{theorem}
\begin{proof}
To start with, for any $\bv_h \in V_h =\text{BDM}_{k-1}$, there exists
%$\sig \in \boldsymbol{H}^2(\Omega;\mathbb{S})$
$\sig \in H^1(\Omega;\mathbb{S})$
such that \cite{costabel2010bogovskiui}
\begin{align}\label{divindentity}
\div{\sig}=\bv_h, \qquad\Vert \sig \Vert_1\lesssim \Vert \bv_h\Vert. 
\end{align}
Moreover, since $\bv_h\in \text{BDM}_{k-1}\subset H(\div,\Omega;\mathbb{V})$, we have 
$$
\sig \in H^1(\Omega;\mathbb{S})\cap H(\div \div,\Omega;\mathbb{S}).
$$
Now, let us define $\sig_h\in \Sigma_h$ such that all the DOF \eqref{DOF0}-\eqref{DOF6} vanish except
\begin{align}
\left(\boldsymbol{n}^{\top} \boldsymbol{\sigma_h} \boldsymbol{n}, q\right)_F =\left(\boldsymbol{n}^{\top} \boldsymbol{\sigma} \boldsymbol{n}, q\right)_F &\quad \forall q \in \P_{k-d}(F),
\quad F \in \mathcal{F}^1(K), \label{DOF8}  \\
\left(\Pi_F \boldsymbol{\sigma_h} \boldsymbol{n}, \boldsymbol{q}\right)_F=\left(\Pi_F \boldsymbol{\sigma} \boldsymbol{n}, \boldsymbol{q}\right)_F &\quad \forall \boldsymbol{q} \in \mathrm{ND}_{k-2}(F), \quad F \in \mathcal{F}^1(K), \label{DOF9}  \\
\left(\boldsymbol{n}^{\top} \operatorname{div} \boldsymbol{\sigma_h}, q\right)_F =\left(\boldsymbol{n}^{\top}\bv_h,  q\right)_F & \quad \forall q \in \P_{k-1}(F), 
\quad F \in \mathcal{F}^1(K), \label{DOF10} \\
(\div \div \boldsymbol{\sigma_h}, q)_K=-(\bv_h, \nabla q)_K& \quad \forall q \in \P_{k-2}(K) / \P_1(K),  \label{DOF11}\\
(\div \boldsymbol{\sigma_h}, \boldsymbol{q})_K=(\bv_h, \boldsymbol{q})_K& \quad \forall \boldsymbol{q} \in\left(\P_{k-3}(K ; \mathbb{K}) / \P_0(K ; \mathbb{K})\right) x. \label{DOF12}
\end{align}
Noting that $\bv_h\in \text{BDM}_{k-1}$, a standard scaling argument shows that the sum of all the DOF of $\bv_h$ can be bounded by $\Vert \bv_h\Vert$.
Further, recall that the DOF of $\bv_h\in \text{BDM}_{k-1}$ are defined by
\begin{align}
\left(\boldsymbol{n}^{\top}\bv_h,  q\right)_F & \quad \forall q \in \P_{k-1}(F), \quad  \P_{k-1}(F),\quad F\in \partial K,\label{BDMDOF0}\\
\left(\bv_h,  q\right)_K & \quad \forall q \in \mathrm{ND}_{k-3}(K)=\text{grad} \P_{k-2}(K) \oplus \P_{k-3}(K ; \mathbb{K}) x. \label{BDMDOF1}
\end{align}
Then, again using a scaling argument and \eqref{divindentity}, we obtain 
\begin{align}\label{inf1} 
\Vert \sig_h\Vert \lesssim  \Vert \sig \Vert_1+ \Vert \bv_h\Vert   \lesssim  \Vert \bv_h\Vert.
\end{align}
Next, we prove that 
\begin{align}\label{inf2}
\div \boldsymbol{\sigma_h}=\bv_h.
\end{align}
Due to our construction \eqref{DOF8}--\eqref{DOF12} and in view of the DOF of $\text{BDM}_{k-1}$ being defined by \eqref{BDMDOF0}--\eqref{BDMDOF1}, we only need to show that 
\begin{align}
\left(\boldsymbol{n}^{\top} \operatorname{div} \boldsymbol{\sigma_h}, q\right)_F =\left(\boldsymbol{n}^{\top}\bv_h,  q\right)_F & \quad \forall q \in \P_{k-1}(F), F\in \partial K, \label{BDMDOF3}\\
\left(\div \sig_h,  \boldsymbol{q} \right)_K=\left(\bv_h,  \boldsymbol{q}\right)_K & \quad \forall \boldsymbol{q} \in \P_{k-3}(K ; \mathbb{K}) x, \label{BDMDOF4}\\
\left(\div \sig_h,  \boldsymbol{q} \right)_K=\left(\bv_h,  \boldsymbol{q}\right)_K & \quad \forall \boldsymbol{q} \in \text{grad} \P_{k-2}(K). \label{BDMDOF5}
\end{align}
in order to conclude \eqref{inf2}.
First, \eqref{BDMDOF3} is obvious due to the DOF \eqref{DOF10} given in the definition of $\sig_h$. 
Second, due to  the DOF \eqref{DOF12} given in the definition of $\sig_h$, we have 
\begin{align}
\left(\div \sig_h,  \boldsymbol{q}\right)_K=\left(\bv_h,  \boldsymbol{q}\right)_K & \quad \forall \boldsymbol{q} \in \left(\P_{k-3}(K ; \mathbb{K}) / \P_0(K ; \mathbb{K})\right) x. \label{BDMDOF6}
\end{align}
In addition, for any $\boldsymbol{q} \in \P_0(K ; \mathbb{K}) x$, noting \eqref{divindentity} and $\eps(\boldsymbol{q})=0$, integration by parts yields
\begin{align}
\left(\div \sig_h-\bv_h,  \boldsymbol{q}\right)_K&=\left(\div \sig_h-\div \sig,  \boldsymbol{q}\right)_K=\left( (\sig_h- \sig)\boldsymbol{n},  \boldsymbol{q}\right)_{\partial K}-\left(\sig_h-\sig,  \eps(\boldsymbol{q})\right)_K\nonumber\\
&=\left( (\sig_h-\sig)\boldsymbol{n},  \boldsymbol{q}\right)_{\partial K}.
\end{align}
Now, taking advantage of the DOF \eqref{DOF8}-\eqref{DOF9} used in the definition of $\sig_h$, we obtain
\begin{align}\label{BDMDOF7}
\left(\div \sig_h-\bv_h,  \boldsymbol{q}\right)_K=\left( (\sig_h-\sig)\boldsymbol{n},  \boldsymbol{q}\right)_{\partial K}=0, \quad \boldsymbol{q} \in \P_0(K ; \mathbb{K}) x.
\end{align}
Combining \eqref{BDMDOF6} and \eqref{BDMDOF7} gives \eqref{BDMDOF4}. 

Next, for any $\boldsymbol{q}\in \text{grad} (\P_{k-2}(K)/\P_1(K))$, namely $\boldsymbol{q}=\nabla p$ for some $p\in \P_{k-2}(K)/\P_1(K)$, using integration by parts, we infer 
\begin{align}
\left(\div \sig_h-\bv_h,  \boldsymbol{q}\right)_K&=\left(\div \sig_h,  \nabla p\right)_K-\left(\bv_h,  \nabla p\right)_K \nonumber\\
&=\left(\boldsymbol{n}^\top\div \sig_h,  p\right)_{\partial K}-\left(\div\div\sig_h,  p\right)_K-\left((\boldsymbol{n}^{\top}\bv_h, p)_{\partial K}-(\bv_h, \nabla p)_K\right).
\end{align}
Then, using \eqref{DOF10}--\eqref{DOF11} and noting that $p\in \P_{k-2}(K)/\P_1(K)$ implies $p|_{\partial K}\in \P_{k-1}(\partial K)$, we find
\begin{align}
\left(\div \sig_h-\bv_h,  \boldsymbol{q}\right)_K&=\left(\boldsymbol{n}^\top\div \sig_h,  p\right)_{\partial K}-\left(\div\div\sig_h,  p\right)_K-\left((\boldsymbol{n}^{\top}\bv_h, p)_{\partial K}-(\bv_h, \nabla p)_K\right)\nonumber\\
&=\left(\boldsymbol{n}^\top\bv_h,  p\right)_{\partial K}-\left(\div\div\sig_h,  p\right)_K-\left((\boldsymbol{n}^{\top}\bv_h, p)_{\partial K}-(\bv_h, \nabla p)_K\right)\nonumber\\
&=0,\quad \forall \boldsymbol{q} \in \text{grad} \left(\P_{k-2}(K)/\P_1(K)\right) \label{BDMDOF8}.
\end{align}
Moreover, for any $\boldsymbol{q}\in \text{grad} \P_1(K)$, namely $\boldsymbol{q}=\nabla p$ for some $p\in P_1(K)$, noting \eqref{divindentity}, and using integration by parts as well as
$\eps(\boldsymbol{q})=\eps(\nabla p)=0$, we conclude
\begin{align}
\left(\div \sig_h-\bv_h,  \boldsymbol{q}\right)_K&=\left(\div \sig_h-\div \sig,  \boldsymbol{q}\right)_K=\left( (\sig_h- \sig)\boldsymbol{n},  \boldsymbol{q}\right)_{\partial K}-\left(\sig_h-\sig,  \eps(\boldsymbol{q})\right)_K\nonumber\\
&=\left( (\sig_h-\sig)\boldsymbol{n},  \boldsymbol{q}\right)_{\partial K}.
\end{align}
Due to \eqref{DOF8}--\eqref{DOF9}, we therefore obtain
\begin{align}\label{BDMDOF8}
\left(\div \sig_h-\bv_h,  \boldsymbol{q}\right)_K=\left( (\sig_h-\sig)\boldsymbol{n},  \boldsymbol{q}\right)_{\partial K}=0, \quad \forall \boldsymbol{q} \in \text{grad}\P_1(K).
\end{align}
Combining \eqref{BDMDOF7} and \eqref{BDMDOF8} gives \eqref{BDMDOF5}.

But then, with \eqref{inf1} and \eqref{inf2} proven, the inf-sup condition \eqref{eq:inf_sup_sigma_h_v_h}
obviously is true.
\end{proof}
\begin{remark}Let both
%$\Pi_h: \boldsymbol{H}^2(\Omega; \mathbb{S})\rightarrow \Sigma_h$ 
$\Pi_h: H^2(\Omega;\mathbb{S})\rightarrow \Sigma_h$
%and $\Pi_h^{\div}: \boldsymbol{H}^1(\Omega)\rightarrow \text{BDM}_{k-1}$
and $\Pi_h^{\div}: H^1(\Omega;\mathbb{V})\rightarrow \text{BDM}_{k-1}$
be the canonical interpolations. Then from the proof of Theorem \ref{lem:inf_sup}, we have that for any
%$\sig\in \boldsymbol{H}^2(\Omega;\mathbb{S})$, 
$\sig\in H^2(\Omega; \mathbb{S})$, 
\begin{equation}\label{commu_Idnetity}
\div \Pi_h \sig=\Pi_h^{\div}\div \sig. 
\end{equation}

\end{remark}
The second basic discrete inf-sup-condition we are going to use is well known, see, e.g.,~\cite{Boffi2013mixed}, and reads: There exists a constant $\beta_4$ which does not depend on the mesh size $h$ such that
\begin{equation}\label{eq:inf_sup_w_h_q_h}
\exists \beta_{4} > 0: \quad \inf_{\substack{q_h \in P_h \\ q_h \neq 0}} \sup_{\substack{\bw_h \in V_h \\ \bw_h \neq \bm 0}} 
\frac{(\div \bw_h,q_h)}{\Vert \div \bw_h \Vert_{H(\div,\Omega;\mathbb{V})} \Vert q_h \Vert} \ge \beta_{4} > 0 .
\end{equation} 

With the conditions~\eqref{eq:inf_sup_sigma_h_v_h} and~\eqref{eq:inf_sup_w_h_q_h} at hand, we are able to establish the well-posedness of problem~\eqref{eq:Biot_fem}.
\begin{theorem}\label{thm:well-posedness_h}
Problem~\eqref{eq:Biot_fem} is well-posed.
% i.e., it has a unique solution $(\sig_h,(\bu_h,\bw_h),p_h)^T \in X_h$. 
In particular, the continuity estimate~\eqref{eq:boundedness} holds on $X_h \times X_h$ and the bilinear form 
$A ((\cdot,(\cdot,\cdot),\cdot),(\cdot,(\cdot,\cdot),\cdot))$ satisfies the inf-sup condition
%\begin{align}
%\mathcal A ((\tsig_h,(\tbu_h,\tbw_h),\tp_h),(\tau_h,(\bv_h,\bz_h),q_h))\le C_1  \left(\Vert \tsig_h \Vert_\Sigma + \Vert (\tbu_h,\tbw_h)^T \Vert_{\bar{V}} + \Vert \tp_h \Vert\right) (\Vert \bm \tau_h \Vert_\Sigma + \Vert (\bv_h,\bz_h)^T \Vert_{\bar{V}} + \Vert q_h \Vert)
%\end{align}
%for all $(\tsig_h,(\tbu_h,\tbw_h),\tp_h)^T,(\tau_h,(\bv_h,\bz_h),q_h)^T \in X_h$ and
\begin{align}\label{inf-sup-dis}
\sup_{\substack{\bm \tau_h \in \Sigma_h, \\ (\bv_h,\bz_h)\in \bar{V}_h, \\ q_h\in P_h}}\frac{\mathcal A ((\tsig_h,(\tbu_h,\tbw_h),\tp_h),(\bm \tau_h,(\bv_h,\bz_h),q_h))}{\Vert \bm \tau_h \Vert_\Sigma + \Vert (\bv_h,\bz_h)^T \Vert_{\bar{V}} + \Vert q_h \Vert}\ge c_1 \left(\Vert \tsig_h \Vert_\Sigma + \Vert (\tbu_h,\tbw_h)^T \Vert_{\bar{V}} + \Vert \tp_h \Vert\right),
\end{align}
for all  $(\sig_h,(\bu_h,\bw_h),p_h)^T \in X_h$. The constants $C_1$  and $c_1$ neither depend on the
parameter $\lambda$ nor on the mesh size $h$. 
\end{theorem}
\begin{proof}
The boundedness of the bilinear form $A ((\cdot,(\cdot,\cdot),\cdot),(\cdot,(\cdot,\cdot),\cdot))$, i.e., the estimate~\eqref{eq:boundedness}, on $X_h \times X_h$ follows from the conformity of the method, i.e., the fact that $X_h \subset X$, together with~\eqref{eq:boundedness}.

To repeat the chain of arguments from the proof of Theorem~\ref{thm:well-posedness}, we have to show that for an arbitrary element $\by_h = (\bm \tau_h,(\bv_h,\bz_h),q_h)^T \in X_h$ we can find an element $\tbx_{h,0} = (\tsig_h,(\tbu_h,\tbw_h), \tp_h)^T \in X_h$ such that the estimates
\begin{equation}\label{eq:boundedness_estimate_h}
\Vert \tsig_h \Vert_\Sigma^2 + \Vert (\tbu_h,\tbw_h)^T \Vert_{\bar{V}}^2 + \Vert \tp_h \Vert^2 
\le c_{3} (\Vert \bm \tau_h \Vert_\Sigma^2 + \Vert (\bv_h,\bz_h)^T \Vert_{\bar{V}}^2 + \Vert q_h \Vert^2)
\end{equation}
and
\begin{equation}\label{eq:coercivity_estimate_h}
\mathcal A ((\tsig_h,(\tbu_h,\tbw_h),\tp_h),(\tau_h,(\bv_h,\bz_h),q_h)) 
\ge c_{4} (\Vert \bm \tau_h \Vert_\Sigma^2 + \Vert (\bv_h,\bz_h)^T \Vert_{\bar{V}}^2 + \Vert q_h \Vert^2) 
\end{equation}
hold true for positive constants $c_{3}$ and $c_{4}$.
%because these two conditions imply the inf-sup condition
%%
%\begin{align}\label{eq:discrete_inf_sup}
%\sup_{\bx_h \in X_h} \frac{\mathcal{A}(\bx_h^T,\by_h^T)}{\Vert \bx_h \Vert_{X}}
%& \ge  \frac{\mathcal{A}(\tbx_h^T,\by_h^T)}{\Vert \tbx_h \Vert_{X}}
%\ge \frac{c_4 \Vert \by_h \Vert_{X}^2}{\Vert \tbx_h \Vert_{X}}
%\ge \frac{c_4}{\sqrt{c_3}} \Vert \by_h \Vert_{X} , \mbox{ for all } \by_h \in X_h
%%= (\tau_h,(\bv_h,\bz_h),q_h)^T \in X_h
%\end{align}
%%
%that we have to prove.
%
So let $\by_h = (\bm \tau_h,(\bv_h,\bz_h),q_h)^T \in X_h$ be arbitrary but fixed.
Then, due to inf-sup condition~\eqref{eq:inf_sup_sigma_h_v_h} there exists an element $\sig_{0,h} \in \Sigma_h$ such that
\begin{equation}\label{eq:def_sig_0}
\div \sig_{0,h} = \bv_h \mbox{ and } \Vert \sig_{0,h} \Vert_\Sigma \le \beta_{3}^{-1} \Vert \bv \Vert ,
\end{equation}
and due to~\eqref{eq:inf_sup_w_h_q_h}
%the inf-sup condition~\eqref{eq:inf_sup_w_h_q_h}
an element $\bw_{0,h} \in V_h$ such that
%there exists an element $\bw_{0,h} \in V_h$ such that
%
\begin{equation}\label{eq:def_w_0}
\div \bw_{0,h} = q \mbox{ and } \Vert \bw_{0,h} \Vert_{H(\div,\Omega;\mathbb{V})} \le \beta_{4}^{-1} \Vert q_h \Vert .
\end{equation}
Analogously to the proof of Theorem~\ref{thm:well-posedness}, we define $\tbx_{h,0} := (\tsig_h,(\tbu_h,\tbw_h), \tp_h)^T \in X_h$ in terms of
\begin{subequations}\label{eq:def_x_tilde_h}
\begin{align}   
\tsig_h & := \sig_{0,h} + \delta_2 \btau_h , \label{eq:def_sig_tilde} \\
\tbu_h & := - \delta_2 \bv_h + \div \btau_h , \label{eq:def_u_tilde}  \\
\tbw_h & := \bw_{0,h} - \div \btau_h - \delta_2 \bz_h , \label{eq:def_w_tilde}  \\
%%(\tbu,\tbw)
\tp_h & : = \delta_2 q_h + \div ( \bv_h + \bz_h ) . \label{eq:def_p_tilde}
\end{align}
\end{subequations}
such that
$$
\tbu_h + \tbw_h = - \delta_2 \bv_h + \bw_{0,h} - \delta_2 \bz_h 
$$
for a positive constant $\delta_2 < \infty$, which we are free to choose. Then, in the same way as in the proof of Theorem~\ref{thm:well-posedness}, we conclude
\begin{align*} 
\Vert \tsig_h \Vert_\Sigma^2 + \Vert (\tbu_h,\tbw_h)^T \Vert_{\bar{V}}^2 + \Vert \tp_h \Vert^2 
& \le c_{3} (\Vert \bm \tau_h \Vert_\Sigma^2 + \Vert (\bv_h,\bz_h)^T \Vert_{\bar{V}}^2 + \Vert q_h \Vert^2)
\end{align*}
where $c_{3}$ depends only on $\delta_2$.

On the other hand, for the above choice~\eqref{eq:def_x_tilde_h} of $\tbx_{h,0} = (\tsig_h,(\tbu_h,\tbw_h), \tp_h)^T$, by following exactly the same steps as in the proof of Theorem~\ref{thm:well-posedness}, we obtain the estimate~\eqref{eq:coercivity_estimate_h} for some positive constant~$c_{4}$. 

{
None of the constants depend on $\lambda$ (for the same reason as in
Theorem~\ref{thm:well-posedness}) and none of them depend on $h$ because neither $\beta_3$ nor $\beta_4$ depend
on $h$. This completes the proof.}
\end{proof}

\begin{remark}
Note that the constants $c_{3}$ and $c_{4}$ in the proof of Theorem~\ref{thm:well-posedness_h} in general differ from $c_1$ and $c_2$ used in the proof of Theorem~\ref{thm:well-posedness} because they depend on the discrete inf-sup constants $\beta_{3}$ and $\beta_{4}$ in~\eqref{eq:inf_sup_sigma_h_v_h} and~\eqref{eq:inf_sup_w_h_q_h}, which in general differ from the constants in~\eqref{eq:inf_sup_sigma_v} and~\eqref{eq:inf_sup_w_q}.
\end{remark}

\begin{remark}
The conformity property \eqref{eq:conformity_Sigma_h}
that the space $\Sigma_h$ satisfies by construction, cf.~\eqref{eq:Sigma_h}--\eqref{eq:div_Sigma_h}, is essential because we want to choose $\tbu_h$ and $\tbw_h $ in $H(\div , \Omega ; \mathbb{V})$, i.e., in $\text{BDM}_{k-1}$, and hence, in order to make them well-defined through the relations~\eqref{eq:def_u_tilde}--\eqref{eq:def_w_tilde},
which resembles the construction on the continuous level, we require the condition $\div \btau_h \subset \text{BDM}_{k-1} \subset H(\div , \Omega ; \mathbb{V})$, i.e., the condition~\eqref{eq:conformity_Sigma_h}.
\end{remark}

\begin{remark}
As the mixed-mixed form of the Stokes problem is a special case of problem \eqref{eq:Biot_weak}, see Remark \ref{Biot_Stokes},  the discretization method developed in this section can also be applied to the Stokes problem. Similar error estimates as those presented in the next subsection are also valid for the corresponding discretization of the Stokes problem. 
\end{remark}

\subsection{A priori error estimates}\label{subsec:fem_error_estimate} 

%The stability estimates from the previous section, i.e., the discrete inf-sup conditions~\eqref{eq:inf_sup_sigma_h_v_h} and~\eqref{eq:inf_sup_w_h_q_h},
The stability estimates from the previous section, in particular the discrete inf-sup condition~\eqref{inf-sup-dis}, not only provide a basis for well-posedness of the formulation~\eqref{eq:Biot_fem} but also allow us to prove a priori error estimates,
as presented in this section.

We start our error analysis with proving a near-best approximation result that is stated in the following Theorem~\ref{thm:near_best_approx}. This theorem includes a suboptimal a priori error estimate as well, which we will improve in a second step by exploiting the specific properties of the canonical interpolation operator
into the BDM space together with known interpolation error estimates. The resulting optimal a priori error
estimates are then presented in the subsequent Theorem~\ref{thm:near_best_approx0}.

\begin{theorem}\label{thm:near_best_approx}
Let $X:= \Sigma^0 \times \bar{V}^0 \times L^2(\Omega)$ and $X_h$ be defined by~\eqref{eq:X_h}.
Then there exists a constant $C >0$ such that the solutions $(\sig_h,(\bu_h,\bw_h),p_h)^T \in X_h$ of
Problem~\eqref{eq:Biot_fem} and $(\sig,(\bu,\bw),p)^T \in X$ of Problem~\eqref{eq:Biot_weak} satisfy
the estimate
\begin{align}\label{eq:near_best_approx}
\Vert \sig - \sig_h \Vert_{H(\div,\Omega;\mathbb{S})} &+ \Vert \bar{\bu} - \bar{\bu}_h \Vert_{\bar{V}} + \Vert p - p_h \Vert  \nonumber  \\
\le &C \inf_{\btau_h \in \Sigma_h,\bar{\bv}_h \in \bar{V}_h, q_h \in Q_h}
  \left(
\Vert \sig - \btau_h \Vert_{H(\div,\Omega;\mathbb{S})} \right. \left.+ \Vert \bar{\bu} - \bar{\bv}_h \Vert_{\bar{V}} + \Vert p - q_h \Vert
\right).
\end{align} 
Furthermore, we have 
\begin{align}\label{eq:near_best_approx_suboptimal}
 \Vert \bar{\bu} - \bar{\bu}_h \Vert_{\bar{V}} +\Vert p - p_h \Vert \le C h^{k-1}\left(h^2|\sig|_k+h|\div\sig|_k+h|\bar{\bu}|_k+|\div\bar{\bu}|_{k-1}+|p|_{k-1}\right).
 \end{align}
\end{theorem}
\begin{proof}
Subtracting~\eqref{eq:Biot_fem} from~\eqref{eq:Biot_weak} results in
\begin{subequations}\label{eq:Biot_error}
\begin{align}   
(A (\sig-\sig_h), \bm \tau_h) + ((\bu-\bu_h),\div \bm \tau_h) = &  0, \label{eq:Biot_error1}\\
(\div (\sig-\sig_h),\bv_h) + ((p-p_h), \div \bv_h) = & 0,  \label{eq:Biot_error2}\\  
- (R_p^{-1}  (\bw-\bw_h), \bz_h) + ((p-p_h),\div \bz_h) = &  0, \label{eq:Biot_error3} \\ 
(\div (\bu-\bu_h), q_h) + (\div (\bw-\bw_h), q_h) + (S (p-p_h),q_h) = & 0  \label{eq:Biot_error4} 
\end{align}
\end{subequations}
for all $(\bm \tau_h,(\bv_h,\bz_h),q_h)^T \in X_h$.

Now, let $\sig_I \in \Sigma_h$, $\bu_I \in V_h$, $\bw_I \in V_h$, and $p_I \in P_h$ be arbitrary. Then, in view of~\eqref{eq:Biot_error}, we have
\begin{subequations}\label{eq:Biot_error_est}
\begin{align}   
(A (\sig_I-\sig_h), \bm \tau_h) + (\bu_I-\bu_h,\div \bm \tau_h) = &  (A (\sig_I-\sig), \bm \tau_h) -(\bu-\bu_I,\div \bm \tau_h), \label{eq:Biot_error_est1}\\
(\div (\sig_I-\sig_h),\bv_h) + (p_I-p_h, \div \bv_h) = & (\div (\sig_I-\sig),\bv_h) -(p-p_I, \div \bv_h),  \label{eq:Biot_error_est2}\\  
- (R_p^{-1}  (\bw_I-\bw_h), \bz_h) + (p_I-p_h,\div \bz_h) = &  - (R_p^{-1}  (\bw_I-\bw), \bz_h) - (p-p_I,\div \bz_h), \label{eq:Biot_error_est3} \\ 
(\div (\bu_I-\bu_h), q_h) + (\div (\bw_I-\bw_h), q_h) + & (S (p_I-p_h),q_h) 
=  -(\div (\bu-\bu_I), q_h) \nonumber \\
- & (\div (\bw -\bw_I), q_h) - (S (p-p_I),q_h)  \label{eq:Biot_error_est4} 
\end{align}
\end{subequations}
for all $(\bm \tau_h,(\bv_h,\bz_h),q_h)^T \in X_h$.

With the definitions $\bar{\bu}_h :=(\bu_h,\bw_h)$, $\bar{\bv}_h :=(\bv_h,\bz_h)$ and $\bar{\bu}_I :=(\bu_I,\bw_I)$, $\bar{\bv}_I :=(\bv_I,\bz_I)$, from the identity~\eqref{eq:Biot_error_est}, by using the discrete inf-sup condition~\eqref{inf-sup-dis} and Cauchy-Schwarz inequality, we obtain
\begin{align*}
\Vert \sig_I - \sig_h \Vert_{H(\div,\Omega;\mathbb{S})} &+ \Vert \bar{\bu}_I - \bar{\bu}_h \Vert_{\bar{V}} +  \Vert p_I - p_h \Vert \\
\le &
c_1^{-1} \left\lvert
\frac{(A (\sig_I-\sig), \bm \tau_h) - (\bu-\bu_I,\div \bm \tau_h)}
{\Vert \btau_h \Vert_{H(\div,\Omega;\mathbb{S})}+ \Vert \bar{\bv}_h \Vert_{\bar{V}} + \Vert q_h \Vert} \right. 
+  \frac{(\div (\sig_I-\sig),\bv_h) - (p-p_I, \div \bv_h)}
{\Vert \btau_h \Vert_{H(\div,\Omega;\mathbb{S})}+ \Vert \bar{\bv}_h \Vert_{\bar{V}} + \Vert q_h \Vert}  \\
+ & \frac{-(R_p^{-1}  (\bw_I-\bw), \bz_h) - (p-p_I,\div \bz_h)}
{\Vert \btau_h \Vert_{H(\div,\Omega;\mathbb{S})}+ \Vert \bar{\bv}_h \Vert_{\bar{V}} + \Vert q_h \Vert}  \\
+ & \left. \frac{(-\div (\bu-\bu_I), q_h) - (\div (\bw -\bw_I), q_h) - (S (p-p_I),q_h)}
{\Vert \btau_h \Vert_{H(\div,\Omega;\mathbb{S})}+ \Vert \bar{\bv}_h \Vert_{\bar{V}} + \Vert q_h \Vert} \right\rvert \\
\le &
c_1^{-1}  \left(
\frac{\lvert(A (\sig_I-\sig), \bm \tau_h)\rvert}{\Vert \btau_h \Vert_{H(\div,\Omega;\mathbb{S})}}
+ \frac{\lvert(\bu-\bu_I,\div \bm \tau_h)\rvert}{\Vert \btau_h \Vert_{H(\div,\Omega;\mathbb{S})}} \right. 
+  \frac{\lvert(\div (\sig_I-\sig),\bv_h)\rvert}{\Vert \bar{\bv}_h \Vert_{\bar{V}}}\\
+& \frac{\lvert(p-p_I, \div (\bv_h + \bz_h))\rvert}{\Vert \bar{\bv}_h \Vert_{\bar{V}}} 
+  \frac{\lvert(R_p^{-1}(\bw_I-\bw), \bz_h)\rvert}{\Vert \bar{\bv}_h \Vert_{\bar{V}}} \\
+ & \left. \frac{\lvert(\div (\bu-\bu_I) + \div (\bw-\bw_I), q_h)\rvert}{\Vert q_h \Vert}
+ \frac{\lvert (S (p-p_I),q_h)\rvert}{\Vert q_h \Vert} \right) \\
\le & c_1^{-1}C_0
\left(
\Vert \sig - \sig_I \Vert_{H(\div,\Omega;\mathbb{S})}
+ \Vert \bu - \bu_I \Vert_V 
+ R_p^{-1} \Vert \bw - \bw_I \Vert_V \right. \\
+ &  \left. \Vert \div (\bu-\bu_I) + \div (\bw-\bw_I) \Vert + (1+S) \Vert p-p_I \Vert \right) \\
\le &  c_1^{-1}C_0
 \left(
\Vert \sig - \sig_I \Vert_{H(\div,\Omega;\mathbb{S})}
+ (1+R_p^{-1}) \Vert \bar{\bu} - \bar{\bu}_I \Vert_{\bar{V }}
+ (1+S) \Vert p-p_I \Vert \right).
\end{align*}
Finally, by using the triangle inequality together with the previous estimate, we conclude
\begin{align*}
\Vert \sig - \sig_h \Vert_{H(\div,\Omega;\mathbb{S})} &+ \Vert \bar{\bu} - \bar{\bu}_h \Vert_{\bar{V}} + \Vert p - p_h \Vert\\
\le &
\Vert \sig - \sig_I \Vert_{H(\div,\Omega;\mathbb{S})} 
+ \Vert \sig_I - \sig_h \Vert_{H(\div,\Omega;\mathbb{S})} 
+  \Vert \bar{\bu} - \bar{\bu}_I \Vert_{\bar{V}} 
+ \Vert \bar{\bu}_I - \bar{\bu}_h \Vert_{\bar{V}} \\
+ & \Vert p - p_I \Vert
+ \Vert p_I - p_h \Vert \\
\le &
C \inf_{\sig_I \in \Sigma_h, \bar{\bu}_I \in \bar{V}_h, p_I \in Q_h} 
\left( \Vert \sig - \sig_I \Vert_{H(\div,\Omega;\mathbb{S})} \right.
+  \left. \Vert \bar{\bu} - \bar{\bu}_I \Vert_{\bar{V}}
+ \Vert p - p_I \Vert \right),
\end{align*} 
which proves the assertion~\eqref{eq:near_best_approx} because $\sig_I \in \Sigma_h$, $\bar{\bu}_I \in \bar{V}_h$, and $p_I \in P_h$ were arbitrary.
{
The bound~\eqref{eq:near_best_approx_suboptimal} then follows from standard interpolation error estimates.
}
\end{proof}

\begin{remark}
As we can see from the proof of Theorem~\ref{thm:near_best_approx}, the constant $C$ in the
bound~\eqref{eq:near_best_approx} does not depend on the Lam\'{e} parameters, which shows that the discretization
is locking-free.
\end{remark}

\begin{remark}
Choosing for $p_I$ the $L^2$ projection of $p$ onto $P_h$ and repeating the steps of the proof of Theorem~\ref{thm:near_best_approx}, due to cancellation of terms in the estimates, one can also
derive a bound of the form 
\begin{align}\label{eq:near_best_approx_2}
\Vert \sig - \sig_h \Vert_{H(\div,\Omega;\mathbb{S})} + \Vert \bar{\bu} - \bar{\bu}_h \Vert_{\bar{V}} 
\le C \inf_{\btau_h \in \Sigma_h, \bar{\bv}_h \in \bar{V}_h} 
&  \left(
\Vert \sig - \btau_h \Vert_{H(\div,\Omega;\mathbb{S})}+ \Vert \bar{\bu} - \bar{\bv}_h \Vert_{\bar{V}} \right).
\end{align} 
in a similar way.
\end{remark}

Let $Q_h$ denote the $L^2$ projection from $L^2(\Omega)$ to $P_h$ and  $\pi^{\div}_h: H(\div,\Omega;\mathbb{V})\rightarrow V_h$ be the cochain projection defined in \cite{MR3246803} satisfying the commuting property 
\begin{align}
\div \pi^{\div}_h  = Q_h \div.
\end{align}
Then the following theorem holds. 
\begin{theorem}\label{thm:near_best_approx0}
There exists a constant $C >0$ such that the solutions $(\sig_h,(\bu_h,\bw_h),p_h)^T \in X_h$ of
Problem~\eqref{eq:Biot_fem} and $(\sig,(\bu,\bw),p)^T \in X$ of Problem~\eqref{eq:Biot_weak} satisfy the estimate
\begin{align}\label{eq:near_best_approx1}
\Vert \sig - \sig_h \Vert_{H(\div,\Omega;\mathbb{S})} &\nonumber+ \Vert \pi_h^{\div} \bar{\bu} - \bar{\bu}_h \Vert_{\bar{V}} + \Vert Q_hp - p_h \Vert
\le \\
&C \left(\inf_{\btau_h \in \Sigma_h}  
\Vert \sig - \btau_h \Vert_{H(\div,\Omega;\mathbb{S})}+ \Vert \bu- \pi_h^{\div} \bu \Vert+\Vert \bw- \pi_h^{\div} \bw \Vert\right).
\end{align} 
Furthermore, we have 
\begin{align}\label{eq:near_best_approx2}
\Vert \sig - \sig_h \Vert_{H(\div,\Omega;\mathbb{S})}+ \Vert \pi_h^{\div} \bar{\bu} - \bar{\bu}_h \Vert_{\bar{V}} + \Vert Q_hp - p_h \Vert
\le C h^k \left(|\sig|_{k}+|\div\sig|_{k}+|\bu|_{k}+|\bw|_k \right)
\end{align} 
and 
\begin{align}\label{eq:near_best_approx3}
\Vert \bar{\bu} - \bar{\bu}_h \Vert
\le C h^k \left(|\sig|_k+|\div\sig|_k+|\bu|_k+|\bw|_k \right).
\end{align}

\end{theorem}
\begin{proof}
Since $\bu \in H^1(\Omega;\mathbb{V})$ and $\bw \in H(\div,\Omega;\mathbb{V})$, see Remark \ref{regularity},  in the proof of Theorem \ref{thm:near_best_approx}, we can choose $\bar{\bu}_I=\pi^{\div}_h \bar{\bu}=(\pi^{\div}_h \bu, \pi^{\div}_h \bw)$ and $p_I=Q_h p$, and hence, due to $\div V_h=P_h$ and $\div \pi_h^{\div} \bu=Q_h\div \bu$,  obtain
\begin{align*}
&(p-p_I, \div \bv_h)=0, (p-p_I, \div \bz_h)=0, \\
&(\div(\bu-\bu_I), q_h)=0, (\div(\bw-\bw_I), q_h)=0, \\
&(S(p-p_I), q_h)=0,
\end{align*} 
which implies for any $ \sig_I\in \Sigma_h$
\begin{align*}
\Vert \sig_I - \sig_h \Vert_{H(\div,\Omega;\mathbb{S})} + \Vert \bar{\bu}_I - \bar{\bu}_h \Vert_{\bar{V}} +  \Vert p_I - p_h \Vert 
\le &
\tilde{C} \Big\lvert
\frac{(A (\sig_I-\sig), \bm \tau_h) - (\bu-\bu_I,\div \bm \tau_h)}
{\Vert \btau_h \Vert_{H(\div,\Omega;\mathbb{S})}+ \Vert \bar{\bv}_h \Vert_{\bar{V}} + \Vert q_h \Vert} \\
&+  \frac{(\div (\sig_I-\sig),\bv_h) -(R_p^{-1}  (\bw_I-\bw), \bz_h) }
{\Vert \btau_h \Vert_{H(\div,\Omega;\mathbb{S})}+ \Vert \bar{\bv}_h \Vert_{\bar{V}} + \Vert q_h \Vert}\Big\lvert.
\end{align*}
Then by the triangle inequality and Cauchy-Schwarz inequality, we obtain the estimate~\eqref{eq:near_best_approx1}.
The estimate~\eqref{eq:near_best_approx2} follows with the help of the approximation property of $\Sigma_h$ and the projection error estimate for $\pi_h^{\div} \bar{\bu}$, 
see \cite{MR3246803, hong2016uniformly}.
Finally, noting that 
\begin{align}
\Vert \bar{\bu} - \bar{\bu}_h \Vert\le \Vert \bar{\bu} - \pi^{\div}_h\bar{\bu} \Vert+\Vert \pi^{\div}_h\bar{\bu} - \bar{\bu}_h \Vert\le  \Vert \bar{\bu} - \pi^{\div}_h\bar{\bu} \Vert+\Vert \pi^{\div}_h\bar{\bu} - \bar{\bu}_h \Vert_{\bar{V}},
\end{align} 
we infer~\eqref{eq:near_best_approx3}.
\end{proof}

\subsection{Post-processing and improved error estimate for pressure}\label{subsec:fem_post_processing} 
Following the ideas discussed in \cite{arnold1985mixed,bank2019superconvergent,bramble1989local,brandts1994superconvergence,stenberg1991postprocessing,ma2021superconvergence}, we can improve the pressure approximation by a simple post-processing procedure, which increases its convergence order and makes it comparable to that of the tensor and flux variables. 
Let $P_h^{\star}$ be the space of discontinuous piecewise polynomials which are one order higher than those defining $P_h$. Let $Q_0, Q_h^\star, Q_0^\perp$ be the $L^2$ projections onto the spaces $P_0, P_h^{\star}$ and $P_0^\perp$, where $P_0$ is the space of piecewise constant functions and $P_0^\perp$ the space orthogonal to $P_0$ such that $P_h^{\star}=P_0\oplus P_0^\perp$. 
Consider now the following problem: Find $p_h^\star\in P_h^{\star}$ such that 
\begin{align}
(\nabla_h p_h^\star, \nabla_h q_h)&=-(R_p^{-1}\bw_h, \nabla_h q_h), ~~\forall~~q_h\in P_0^\perp, \label{post1}\\
(p_h^\star, q_{h,0})&=(p_h, q_{h,0}),~~\forall~~q_h\in P_0\label{post2}.
\end{align}
It is easy to see that $p_h^\star,$ is well-defined by checking that
\eqref{post1}-\eqref{post2} is a nonsingular linear system. The following error analysis of this post-processing procedure is nearly the same as that in existing literature mentioned in this section. Nonetheless, we include a detailed proof in order to keep our presentation self-contained.
\begin{theorem}\label{thm:post}
Let $r-1$ be the polynomial degree of $P_h^{\star}$ and $p_h^* \in P_h^{\star}$ be the solution of \eqref{post1}-\eqref{post2}. Then we have
\begin{align}
\Vert p-p_h^\star\Vert\le C h^r \left(|\sig|_r+|\div\sig|_r+|\bu|_r+|\bw|_r+|p|_r\right).
\end{align}
\end{theorem}
\begin{proof}
To prove $\Vert p-p_h^\star\Vert\le C h^r$, it is sufficient to prove $\Vert Q_h^\star p-p_h^\star\Vert\le C h^r$ because $\Vert Q_h^\star p-p\Vert\le C h^r$. 
Noting that $Q_h^\star =Q_0+Q_0^\perp$, we have 
\begin{align}
Q_h^\star (p-p_h^\star)=Q_0(p-p_h^\star)+Q_0^\perp(p-p_h^\star)\in P_0+P_0^\perp. 
\end{align}
From the facts that $Q_0Q_hp=Q_0p, Q_0p_h^\star=Q_0p_h$, we have 
$$
Q_0(p-p_h^\star)=Q_0(Q_hp-p_h). 
$$
Recall that from \eqref{eq:near_best_approx2}, we already have $\Vert Q_hp-p_h\Vert\le Ch^r$. Hence we only need to prove 
\begin{align}\label{perp1}
\Vert Q_0^\perp(p-p_h^\star) \Vert \le Ch^r.
\end{align}
From \eqref{Biot3_strong}, we have 
\begin{align}
\nabla p&=-R_p^{-1}  \bw \end{align}
 and further 
 \begin{align}\label{post:true}
(\nabla_h p, \nabla_h q_h)&=-(R_p^{-1}\bw, \nabla_h q_h), ~~\forall~~q_h\in P_0^\perp.
\end{align}
Therefore, by subtracting \eqref{post1} from \eqref{post:true}, we obtain
$$
(\nabla_h (p-p_h^\star), \nabla_h q_h)=-(R_p^{-1}(\bw-\bw_h), \nabla_h q_h), ~~\forall~~q_h\in P_0^\perp.
$$
Using the identity
\begin{align}
p-p_h^\star=(p-Q_h^\star p)+(Q_h^\star p-p_h^\star)=(p-Q_h^\star p)+Q_0(Q_h^\star p-p_h^\star)+Q_0^\perp(Q_h^\star p-p_h^\star),
\end{align}
we see
$$
(\nabla_h (Q_0^\perp (p-p_h^\star)), \nabla_h q_h)=-(\nabla_h (p-Q_h^\star p), \nabla_h q_h)-(R_p^{-1}(\bw-\bw_h), \nabla_h q_h), ~~\forall~~q_h\in P_0^\perp.
$$
Hence, we obtain 
\begin{align}
\Vert \nabla_h (Q_0^\perp (p-p_h^\star))\Vert \le \Vert \nabla_h (p-Q_h^\star p)\Vert +R^{-1}_p\Vert \bw-\bw_h\Vert.
\end{align}
Noting that 
$$
\Vert Q_0^\perp (p-p_h^\star)\Vert \le  C h \Vert \nabla_h (Q_0^\perp (p-p_h^\star))\Vert, 
$$
we conclude 
$$
\Vert Q_0^\perp (p-p_h^\star)\Vert \le C h \Vert \nabla_h (Q_0^\perp (p-p_h^\star))\Vert \le C h\Vert \nabla_h (p-Q_h^\star p)\Vert +C h R^{-1}_p\Vert \bw-\bw_h\Vert. 
$$
By the approximation property of $Q_h^\star $ and the approximation estimate \eqref{eq:near_best_approx3}, we finally get 
$$
\Vert Q_0^\perp (p-p_h^\star)\Vert  \le C h\Vert \nabla_h (p-Q_h^\star p)\Vert +C h R^{-1}_p\Vert \bw-\bw_h\Vert\le Ch^r+Ch\cdot h^r,
$$
which implies \eqref{perp1}.
\end{proof}

\section{Conclusions}\label{sec:num_res} 

We have presented a four-field formulation of the time-step equations of the quasi-static Biot's model of consolidation obtained after semi-discretization in time by some implicit time integrator. The physical quantities of interest in our model are the effective stress tensor borne by the solid skeleton, the solid displacement, the fluid flux, and the pore fluid pressure. The model in strong form consists of the constitutive equation of linear elasticity (Hooke's law in compliance form) and the equations describing the balance of linear momentum, Darcy's law for fluid flow through a porous medium, and the conservation of mass. We have proposed a four-field variational formulation in which the above-mentioned physical quantities of interest appear directly as the four fields thus providing a suitable starting point for developing discrete numerical models that are momentum- and mass conservative. Utilizing a novel Hilbert space, we have proven this intrinsic variational problem to be well-posed. Next, we have exploited a family of  $H({\rm div} \, {\rm div}) \cap H({\rm div})$-conforming symmetric tensor-valued finite elements to construct a mixed-mixed discretization that preserves the angular momentum as well as the fluid mass pointwise, i.e., in a strong sense. Further, we have proven the well-posedness of the discrete model resulting from application of this mixed-mixed method to the new four-field formulation. Finally, we have derived optimal a priori error estimates for the numerical approximations of all four fields obtained from this discrete model.

\textbf{Acknowledgements.} The first author Qingguo Hong acknowledges the support from NSF Grant NSF DMS-2419033. 

%\section{Numerical results}\label{sec:num_res} 
%{
%We suppose that the domain $\Omega$ is the unit square in $\mathbb{R}^2$, and during the discretization, it has
%been partitioned as bisections of $2N^2$ triangles with mesh size $h= 1/N$.  All the numerical tests
%included in this section have been carried out in FEniCS. We use the basis of the finite space $\Sigma_h$ provided in \cite{HuMaZhang2021}. 
% We consider the simplest case of a system with only one pressure and one flux, namely, Biot's consolidation model. We solve system \eqref{eq:Biot_strong} for
%
%$$
%\boldsymbol{f}=\left[\begin{array}{r}
%-\left(2 y^3-3 y^2+y\right)\left(12 x^2-12 x+2\right)-(x-1)^2 x^2(12 y-6)+900(y-1)^2 y^2\left(4 x^3-6 x^2+2 x\right) \\
%\left(2 x^3-3 x^2+x\right)\left(12 y^2-12 y+2\right)+(y-1)^2 y^2(12 x-6)+900(x-1)^2 x^2\left(4 y^3-6 y^2+2 y\right)
%\end{array}\right]
%$$
%
%and $g=R_1\left(\frac{\partial \phi_2}{\partial x}+\frac{\partial \phi_2}{\partial y}\right)-\alpha_{p_1}\left(\phi_2-1\right)$, where
%$$
% \phi_1=(x-1)^2(y-1)^2 x^2 y^2, \quad \phi_2=900(x-1)^2(y-1)^2 x^2 y^2,(x, y) \in \Omega .
%$$
%
%
%Then, the exact solution of system \eqref{eq:Biot_strong} with boundary conditions $\left.\boldsymbol{u}\right|_{\partial \Omega}=0,\left.\boldsymbol{v} \cdot \boldsymbol{n}\right|_{\partial \Omega}=0$ is given by
%
%$$
%\boldsymbol{u}=\left(\frac{\partial \phi_1}{\partial y},-\frac{\partial \phi_1}{\partial x}\right), p=\phi_2-1, \boldsymbol{v}=-R_1 \nabla p, \text { where } p \in L_0^2(\Omega) .
%$$
%}
%\subsection{...}
%
%\subsection{...}
%
%\subsection{}

\bibliographystyle{plain}
\bibliography{HKL_4field}

\begin{thebibliography}{10}

\bibitem{Aichele2021fluids}
J.~Aichele and S.~Catheline.
\newblock Fluids alter elasticity measurements: Porous wave propagation
  accounts for shear wave dispersion in elastography.
\newblock {\em Front. Phys.}, 9, 2021.

\bibitem{arnold2008finite}
D.~N. Arnold, G.~Awanou, and R.~Winther.
\newblock Finite elements for symmetric tensors in three dimensions.
\newblock {\em Math. Comput.}, 77(263):1229--1251, 2008.

\bibitem{arnold1985mixed}
D.~N. Arnold and F.~Brezzi.
\newblock Mixed and nonconforming finite element methods: implementation,
  postprocessing and error estimates.
\newblock {\em ESAIM: Math. Model. Numer. Anal.}, 19(1):7--32, 1985.

\bibitem{arnold2006finite}
D.~N. Arnold, R.~S. Falk, and R.~Winther.
\newblock Finite element exterior calculus, homological techniques, and
  applications.
\newblock {\em Acta Numer.}, 15:1--155, 2006.

\bibitem{arnold2002mixed}
D.~N. Arnold and R.~Winther.
\newblock Mixed finite elements for elasticity.
\newblock {\em Numer. Math.}, 92(3):401--419, 2002.

\bibitem{Babuska1971error}
I.~Babu{\v s}ka.
\newblock Error-bounds for finite element method.
\newblock {\em Numer. Math.}, 16:322--333, 1970/71.

\bibitem{Babuska1992locking}
Ivo Babu\v{s}ka and Manil Suri.
\newblock Locking effects in the finite element approximation of elasticity
  problems.
\newblock {\em Numer. Math.}, 62:439--463, 1992.

\bibitem{bank2019superconvergent}
R.~E. Bank and Y.~Li.
\newblock Superconvergent recovery of raviart--thomas mixed finite elements on
  triangular grids.
\newblock {\em J. Sci. Comput.}, 81(3):1882--1905, 2019.

\bibitem{Barenblatt1960basic}
G.I. Barenblatt, G.I. Zheltov, and I.N. Kochina.
\newblock Basic concepts in the theory of seepage of homogeneous liquids in
  fissured rocks [strata].
\newblock {\em J. Appl. Math. Mech.}, 24(5):1286--1303, 1960.

\bibitem{Biot1941general}
M.A. Biot.
\newblock General theory of three-dimensional consolidation.
\newblock {\em J. Appl. Phys.}, 12(2):155--164, 1941.

\bibitem{Biot1956waves1}
M.A. Biot.
\newblock Theory of propagation of elastic waves in a fluid-saturated porous
  solid. {I}. {L}ow-frequency range.
\newblock {\em J. Acoust. Soc. Amer.}, 28:168--178, 1956.

\bibitem{Biot1956waves2}
M.A. Biot.
\newblock Theory of propagation of elastic waves in a fluid-saturated porous
  solid. {II}. {H}igher frequency range.
\newblock {\em J. Acoust. Soc. Amer.}, 28:179--191, 1956.

\bibitem{Biot1962acoust}
M.A. Biot.
\newblock Generalized theory of acoustic propagation in porous dissipative
  media.
\newblock {\em J. Acoust. Soc. Amer.}, 34:1254--1264, 1962.

\bibitem{Biot1962Mech}
M.A. Biot.
\newblock Mechanics of deformation and acoustic propagation in porous media.
\newblock {\em J. Appl. Phys.}, 33(4):1482--1498, April 1962.

\bibitem{Boffi2013mixed}
D.~Boffi, F.~Brezzi, and M.~Fortin.
\newblock {\em Mixed finite element methods and applications}, volume~44 of
  {\em Springer Ser. Comput. Math.}
\newblock Springer, Heidelberg, 2013.

\bibitem{Bothl2022iterative}
J.W. Both, N.A. Barnafi, F.A. Radu, P.~Zunino, and A.~Quarteroni.
\newblock Iterative splitting schemes for a soft material poromechanics model.
\newblock {\em Comput. Methods Appl. Mech. Engrg.}, 388:114183, 2022.

\bibitem{bramble1989local}
J.~H. Bramble and J.~Xu.
\newblock A local post-processing technique for improving the accuracy in mixed
  finite-element approximations.
\newblock {\em SIAM J. Numer. Anal.}, 26(6):1267--1275, 1989.

\bibitem{brandts1994superconvergence}
J.~H. Brandts.
\newblock Superconvergence and a posteriori error estimation for triangular
  mixed finite elements.
\newblock {\em Numer. Math.}, 68(3):311--324, 1994.

\bibitem{Carlson1973Lin}
D.E. Carlson.
\newblock {\em Linear Thermoelasticity}, pages 297--345.
\newblock Springer Berlin Heidelberg, Berlin, Heidelberg, 1973.

\bibitem{chen2020finite}
L.~Chen and X.~Huang.
\newblock Finite elements for divdiv-conforming symmetric tensors.
\newblock {\em arXiv preprint arXiv:2005.01271}, 2020.

\bibitem{chen2022finiteele}
L.~Chen and X.~Huang.
\newblock A finite element elasticity complex in three dimensions.
\newblock {\em Math. Comput.}, 91(337):2095--2127, 2022.

\bibitem{ChenHuang2022}
L.~Chen and X.~Huang.
\newblock Finite elements for div- and divdiv-conforming symmetric tensors in
  arbitrary dimension.
\newblock {\em SIAM J. Numer. Anal.}, 60(4):1932--1961, 2022.

\bibitem{chen2022finite}
L.~Chen and X.~Huang.
\newblock Finite elements for divdiv conforming symmetric tensors in three
  dimensions.
\newblock {\em Math. Comput.}, 91(335):1107--1142, 2022.

\bibitem{ChenHuang2025}
L.~Chen and X.~Huang.
\newblock {A new div-div-conforming symmetric tensor finite element space with
  applications to the biharmonic equation}.
\newblock {\em Math. Comput.}, 94(351):33--72, 2025.

\bibitem{Cheng2016Poro}
A.~H.-D. Cheng.
\newblock {\em Poroelasticity}, volume~27 of {\em Theory and Applications of
  Transport in Porous Media}.
\newblock Springer, [Cham], 2016.

\bibitem{CockburnGopalakrishnan2004characterization}
B.~Cockburn and J.~Gopalakrishnan.
\newblock A characterization of hybridized mixed methods for second order
  elliptic problems.
\newblock {\em SIAM J. Numer. Anal.}, 42(1):283--301, 2004.

\bibitem{costabel2010bogovskiui}
M.~Costabel and A.~McIntosh.
\newblock On bogovski{\v{\i}} and regularized poincar{\'e} integral operators
  for de rham complexes on lipschitz domains.
\newblock {\em Math. Z.}, 265(2):297--320, 2010.

\bibitem{Dai2014comparison}
Z.~Dai, Y.~Peng, H.A. Mansy, R.H. Sandler, and T.J. Royston.
\newblock Comparison of poroviscoelastic models for sound and vibration in the
  lungs.
\newblock {\em J. Vib. Acoust.}, 136(5):0510121--5101211, 2014.

\bibitem{MR3246803}
R.~S. Falk and R.~Winther.
\newblock Local bounded cochain projections.
\newblock {\em Math. Comp.}, 83(290):2631--2656, 2014.

\bibitem{Guo_etal2018subject-specific}
L.~Guo, J.C. Vardakis, T.~Lassila, M.~Mitolo, N.~Ravikumar, D.~Chou, M.~Lange,
  A.~Sarrami-Foroushani, B.J. Tully, Z.A. Taylor, S.~Varma, A.~Venneri, A.F.
  Frangi, and Y.~Ventikos.
\newblock Subject-specific multi-poroelastic model for exploring the risk
  factors associated with the early stages of {A}lzheimer's disease.
\newblock {\em Interface Focus}, 8(1):20170019, 2018.

\bibitem{Haga2012causes}
J.~B. Haga, H.~Osnes, and H.~P. Langtangen.
\newblock On the causes of pressure oscillations in low‐permeable and
  low‐compressible porous media.
\newblock {\em Int. J. Numer. Anal. Methods Geomech.}, 36(12):1507--1522, 2012.

\bibitem{HongHuMaXu2021}
Q.~Hong, J.~Hu, L.~Ma, and J.~Xu.
\newblock New discontinuous {G}alerkin algorithms and analysis for linear
  elasticity with symmetric stress tensor.
\newblock {\em Numer. Math.}, 149(3):645--678, 2021.

\bibitem{hong2016uniformly}
Q.~Hong and J.~Kraus.
\newblock Uniformly stable discontinuous galerkin discretization and robust
  iterative solution methods for the brinkman problem.
\newblock {\em SIAM J. Numer. Anal.}, 54(5):2750--2774, 2016.

\bibitem{HongKraus2018parameter}
Q.~Hong and J.~Kraus.
\newblock Parameter-robust stability of classical three-field formulation of
  {B}iot's consolidation model.
\newblock {\em Electron. Trans. Numer. Anal.}, 48:202--226, 2018.

\bibitem{Hong2019conservativeMPET}
Q.~Hong, J.~Kraus, M.~Lymbery, and F.~Philo.
\newblock Conservative discretizations and parameter-robust preconditioners for
  {B}iot and multiple-network flux-based poroelasticity models.
\newblock {\em Numer. Linear Algebra Appl.}, 2019.
\newblock e2242.

\bibitem{Hong2020ParameterUzawa}
Q.~Hong, J.~Kraus, M.~Lymbery, and F.~Philo.
\newblock Parameter-robust uzawa-type iterative methods for double saddle point
  problems arising in biot’s consolidation and multiple-network
  poroelasticity models.
\newblock {\em Math. Models Methods Appl. Sci.}, 30(13):2523--2555, 2020.

\bibitem{hong2023new}
Q.~Hong, J.~Kraus, M.~Lymbery, and F.~Philo.
\newblock A new practical framework for the stability analysis of perturbed
  saddle-point problems and applications.
\newblock {\em Math. Comput.}, 92(340):607--634, 2023.

\bibitem{HongEtAl2020parameter}
Q.~Hong, J.~Kraus, M.~Lymbery, and M.F. Wheeler.
\newblock Parameter-robust convergence analysis of fixed-stress split iterative
  method for multiple-permeability poroelasticity systems.
\newblock {\em Multiscale Model. Simul.}, 18(2):916--941, 2020.

\bibitem{hu2014finite}
J.~Hu.
\newblock Finite element approximations of symmetric tensors on simplicial
  grids in $\mathbb{R}^n$: The higher order case.
\newblock {\em J. Comput. Math.}, 33(3):283--296, 2015.

\bibitem{HuMaZhang2021}
J.~Hu, R.~Ma, and M.~Zhang.
\newblock A family of mixed finite elements for the biharmonic equations on
  triangular and tetrahedral grids.
\newblock {\em Sci. China Math.}, 64(12):2793--2816, 2021.

\bibitem{hu2025optimality}
J.~Hu, R.~Ma, and M.~Zhang.
\newblock Optimality of adaptive $h(\div\div)$ mixed finite element methods for
  the kirchhoff-love plate bending problem.
\newblock {\em arXiv preprint arXiv:2508.09008}, 2025.

\bibitem{hu2015family}
J.~Hu and S.~Zhang.
\newblock A family of symmetric mixed finite elements for linear elasticity on
  tetrahedral grids.
\newblock {\em Sci. China Math.}, 58(2):297--307, 2015.

\bibitem{hu2016finite}
J.~Hu and S.~Zhang.
\newblock Finite element approximations of symmetric tensors on simplicial
  grids in $\mathbb{R}^n$: The lower order case.
\newblock {\em Math. Models Methods Appl. Sci.}, 26(9):1649--1669, 2016.

\bibitem{KanschatRiviere2018finite}
G.~Kanschat and B.~Riviere.
\newblock A finite element method with strong mass conservation for {B}iot's
  linear consolidation model.
\newblock {\em J. Sci. Comput.}, 77(3):1762--1779, 2018.

\bibitem{KrLeLyOsSc2022}
J.~Kraus, P.~L. Lederer, M.~Lymbery, K.~Osthues, and J.~Sch\"{o}berl.
\newblock Hybridized discontinuous {G}alerkin/hybrid mixed methods for a
  multiple network poroelasticity model with application in biomechanics.
\newblock {\em SIAM J. Sci. Comput.}, 45(6):B802--B827, 2023.

\bibitem{Lee2017parameter}
J.~Lee, K.-A. Mardal, and R.~Winther.
\newblock Parameter-robust discretization and preconditioning of {B}iot's
  consolidation model.
\newblock {\em SIAM J. Sci. Comput.}, 39:A1--A24, 2017.

\bibitem{Lee2016robust}
J.J. Lee.
\newblock Robust error analysis of coupled mixed methods for {B}iot's
  consolidation model.
\newblock {\em J. Sci. Comput.}, 69:610--632, 2016.

\bibitem{Lee2023analysis}
J.J. Lee.
\newblock Analysis and preconditioning of parameter-robust finite element
  methods for {B}iot's consolidation model.
\newblock {\em BIT}, 63(3):Paper No. 42, 23, 2023.

\bibitem{lee2018mixed}
J.J. Lee, E.~Piersanti, K.-A. Mardal, and M.E. Rognes.
\newblock A mixed finite element method for nearly incompressible
  multiple-network poroelasticity.
\newblock {\em SIAM J. Sci. Comput.}, 41(2):A722--A747, 2019.

\bibitem{Lehrenfeld2016high}
C.~Lehrenfeld and J.~Sch\"{o}berl.
\newblock High order exactly divergence-free hybrid discontinuous {G}alerkin
  methods for unsteady incompressible flows.
\newblock {\em Comput. Methods Appl. Mech. Engrg.}, 307:339--361, 2016.

\bibitem{ma2021superconvergence}
L.~Ma.
\newblock Superconvergence of discontinuous galerkin methods for elliptic
  boundary value problems.
\newblock {\em J. Sci. Comput.}, 88(3):62, 2021.

\bibitem{nedelec1980mixed}
J.~C. N{\'e}d{\'e}lec.
\newblock Mixed finite elements in $\mathcal{R}^3$.
\newblock {\em Numer. Math.}, 35(3):315--341, 1980.

\bibitem{Ong2009Deformation}
R.~E. Ong, C.~L. Glisson, S.~D. Herrell, M.~I. Miga, and R.~Galloway.
\newblock {A deformation model for non-rigid registration of the kidney}.
\newblock In Michael~I. Miga and Kenneth~H. Wong, editors, {\em Medical Imaging
  2009: Visualization, Image-Guided Procedures, and Modeling}, volume 7261,
  page 72613A. International Society for Optics and Photonics, SPIE, 2009.

\bibitem{oyarzua2016locking}
R.~Oyarz\'ua and R.~Ruiz-Baier.
\newblock Locking-free finite element methods for poroelasticity.
\newblock {\em SIAM J. Numer. Anal.}, 54:2951--2973, 2016.

\bibitem{Phillips2009overcoming}
P.~J. Phillips and M.~F. Wheeler.
\newblock Overcoming the problem of locking in linear elasticity and
  poroelasticity: an heuristic approach.
\newblock {\em Comput. Geosci.}, 13:5--12, 2009.

\bibitem{Pitre2017}
J.~J. Pitre and J.~L. Bull.
\newblock {\em Imaging the Mechanical Properties of Porous Biological Tissue},
  pages 1--27.
\newblock Springer International Publishing, Cham, 2017.

\bibitem{Showalter2000Diff}
R.E. Showalter.
\newblock Diffusion in poro-elastic media.
\newblock {\em J. Math. Anal. Appl.}, 251, November 2000.

\bibitem{stenberg1991postprocessing}
R.~Stenberg.
\newblock Postprocessing schemes for some mixed finite elements.
\newblock {\em ESAIM: Math. Model. Numer. Anal.}, 25(1):151--167, 1991.

\bibitem{Vardakis2016investigating}
J.C. Vardakis, D.~Chou, B.J. Tully, C.C. Hung, T.H. Lee, P.H. Tsui, and
  Y.~Ventikos.
\newblock Investigating cerebral oedema using poroelasticity.
\newblock {\em Med. Eng. Phys.}, 38(1):48--57, 2016.

\bibitem{Zenisek1984existence}
A.~\v{Z}en\'{i}\v{s}ek.
\newblock The existence and uniqueness theorem in {B}iot's consolidation
  theory.
\newblock {\em Aplikace Mat.}, 29(3):194--211, 1984.

\bibitem{Zenisek1984finite}
A.~\v{Z}en\'{i}\v{s}ek.
\newblock Finite element methods for coupled thermoelasticity and coupled
  consolidation of clay.
\newblock {\em RAIRO Anal. Num\'er.}, 18(2):183--205, 1984.

\bibitem{Yi2014convergence}
S.-Y. Yi.
\newblock Convergence analysis of a new mixed finite element method for
  {B}iot's consolidation model.
\newblock {\em Numer. Methods Partial Differ. Equ.}, 30(4):1189--1210, 2014.

\bibitem{Zienkiewicz1982Basic}
O.C. Zienkiewicz.
\newblock Basic formulation of static and dynamic behaviours of soil and other
  porous media.
\newblock {\em Appl. Math. Mech.}, 3(4):457--468, 1982.

\end{thebibliography}

\end{document}